\documentclass{amsart}

\newtheorem{theorem}{Theorem}[section]
\newtheorem{lemma}[theorem]{Lemma}
\newtheorem{proposition}[theorem]{Proposition}
\newtheorem{corollary}[theorem]{Corollary}
\theoremstyle{definition}
\newtheorem{definition}[theorem]{Definition}
\newtheorem{example}[theorem]{Example}

\theoremstyle{remark}
\newtheorem{remark}[theorem]{Remark}

\numberwithin{equation}{section}

\usepackage{graphicx}
\usepackage{tikz}
\usepackage{tikz-cd}
\usetikzlibrary{arrows,automata}
\usetikzlibrary{shapes,snakes,backgrounds}
\usepackage[dvips]{epsfig}
\usepackage{amsfonts}
\usepackage{amssymb}
\usepackage{amscd}
\usepackage{amstext}
\usepackage{amsmath}
\usepackage{pstricks}
\usepackage{enumerate}
\usepackage{hyperref}
\def\STUFFLE{\mathop{\stuffle}\limits}
\usepackage{ulem}
\def\Lim{\displaystyle\lim}
\def\abs#1{|#1|}
\def\absv#1{\Vert#1\Vert}

\def\shuffle{\mathop{_{^{\sqcup\!\sqcup}}}} 

\def\mod{{\rm\ mod\ }}

\def\adots{\mathinner{\mkern2mu\raise1pt\hbox{.}
\mkern3mu\raise4pt\hbox{.}\mkern1mu\raise7pt\hbox{.}}}

\def\pointir{\unskip . -- \ignorespaces}

\def\up#1{\raise 1ex\hbox{\footnotesize#1}}

\def\H{\mathcal{H}}

\def\span{\mathop\mathrm{span}\nolimits}

\def\path{\rightsquigarrow}

\def\Lyn{{\mathcal Lyn}}

\def\A{\mathcal{A}}

\def\N{{\mathbb N}}
\def\C{{\mathbb C}}
\def\R{{\mathbb R}}
\def\Z{{\mathbb Z}}
\def\Q{{\mathbb Q}}

\def\Q{{\mathbb Q}}

\def\L{\mathrm{L}}
\def\H{\mathrm{H}}

\def\calC{{\mathcal C}}
\def\calE{{\mathcal E}}

\def\calG{{\mathcal G}}
\def\calH{{\mathcal H}}
\def\calL{{\mathcal L}}
\def\calO{{\mathcal O}}
\def\calX{{\mathcal X}}
\def\calZ{{\mathcal Z}}
\def\scal#1#2{\langle #1 | #2 \rangle}

\def\ncs#1#2{#1\langle\langle #2\rangle\rangle}

\def\ncp#1#2{#1\langle #2\rangle}

\def\dext#1#2{#1\{ \!\{ #2\} \!\}}

\def\Li{\mathrm{Li}}

\def\calA{\mathcal{A}}

\gdef\stuffle{\;%
  \setlength{\unitlength}{0.0125cm}%
  \begin{picture}(20,10)(220,580)
  \thinlines
  \put(220,592){\line( 0,-1){ 10}}
  \put(220,582){\line( 1, 0){ 20}}
  \put(240,582){\line( 0, 1){ 10}}
  \put(230,592){\line( 0,-1){ 10}}
  \put(225,587){\line( 1, 0){ 10}}
  \end{picture}\;
}

\newcommand\rsmraise[1]{%
  \ifx#1\displaystyle .8\else
    \ifx#1\textstyle .8\else
      \ifx#1\scriptstyle .6\else
        .45%
      \fi
    \fi
  \fi}
\begingroup
\def\2#1{\ifnum#1<10 0\fi\the#1}
\xdef\isodayandtime{
{\2\day-\2\month-\the\year\space\2{\count0}:%
\2{\count2}}}
\endgroup

\def\poly#1#2{#1\langle #2 \rangle}

\def\QX{\poly{\Q}{X}}

\def\CX{\poly{\C}{X}}

\def\QY{\poly{\Q}{Y}}

\def\CY{\poly{\C}{Y}}

\def\QY_0{\Q\left\langle{Y_0}\right\rangle}

\newcommand*{\longtwoheadrightarrow}{\ensuremath{\joinrel\relbar\joinrel\twoheadrightarrow}}

\begingroup
\def\2#1{\ifnum#1<10 0\fi\the#1}
\xdef\isodayandtime{
{\2\day-\2\month-\the\year\space\2{\count0}:%
\2{\count2}}}
\endgroup

\newcounter{per1}
\definecolor{MyDarkBlue}{rgb}{0,0.08,0.4}

\begin{document}

\title[Families of eulerian functions]{
Familles de fonctions eul\'eriennes impliqu\'ees dans la r\'egularisation de polyz\^etas divergents\\
Families of eulerian functions involved in regularization of divergent polyzetas}

\author{V.C. Bui}
\address{Hue University, 77, Nguyen Hue, Hue, Viet Nam,}
\email{bvchien.vn@gmail.com}

\author{V. Hoang Ngoc Minh}
\address{University of Lille, 1 Place D\'eliot, 59024 Lille, France,\\
LIPN - UMR 7030, CNRS, 93430 Villetaneuse, France,}
\email{vincel.hoang-ngoc-minh@univ-lille.fr,minh@lipn.univ-paris13.fr}
\author{Q. H. Ngo}
\address{Hanoi University of Science and Technology, 1 Dai Co Viet, Hai Ba Trung, Ha Noi, Viet Nam,}
\email{hoan.ngoquoc@hust.edu.vn}
\thanks{The thirst author was supported in part by  Vietnam National Foundation for Science and
Technology Development (NAFOSTED) under grant number 101.04-2021.41.}

\author{V. Nguyen Dinh}
\address{LIPN-UMR 7030, 99 avenue Jean-Baptiste Cl\'ement, 93430 Villetaneuse, France,}
\email{nguyendinh@lipn.univ-paris13.fr}


\date{\today}



\begin{abstract}
Extending the Eulerian functions, we study their relationship with zeta function of several variables.
In particular, starting with Weierstrass factorization theorem (and Newton-Girard identity)
for the complex Gamma function, we are interested in the ratios of $\zeta(2k)/\pi^{2k}$
and their multiindexed generalization, we obtain an analogue situation and draw
some consequences about a structure of the algebra of polyzetas values, by means of some
combinatorics of words and noncommutative rational series. The same frameworks also
allow to study the independence of a family of eulerian functions.\\

{\bf R\'esum\'e:}
En g\'en\'eralisant les fonctions euleriennes, nous \'etudions leurs relations avec
la fonction z\^eta en plusieurs variables. En particulier, \`a partir du th\'eor\`eme de
factorisation de Weierstrass (et l'identit\'e de Newton-Girard) pour la fonction Gamma
complexe, nous nous int\'eressons aux rapports $\zeta(2k)/\pi^{2k}$ et leurs
g\'en\'eralisations. Nous obtenons une situation analogue et nous tirerons quelques
cons\'equences sur une structure de l'alg\`ebre des valeurs polyz\^etas, au moyen de la
combinatoire des mots et des s\'eries rationnelles en variables non commutatifs.
Le m\^eme cadre de travail permet \'egalement d'\'etudier l'ind\'epen\-dance d'une famille
de fonctions euleriennes.
\end{abstract}

\maketitle
\tableofcontents

\section{Introduction}\label{introduction}

Eulerian functions are most significant for analytic number theory
and they are widely applied in Probability theory and in Physical sciences.
They are tightly relating to Riemann zeta functions, for instance as follows
\begin{eqnarray}\label{zeta}
\zeta(s)=\frac1{{\Gamma(s)}}\int_0^{\infty}dt\frac{t^{s-1}}{e^t-1}
&\mbox{and}&{\Gamma(s)}=\int_0^{\infty}du\;u^{s-1}e^{-u},\mbox{ for }\Re(s)>0.
\end{eqnarray}
The function $\Gamma$ is meromorphic, with no zeroes and $-\N^*$ as set of simple poles.
Hence $\Gamma^{-1}$ is entire and admits $-\N^*$ as set of simple zeroes. Moreover,
it satisfies\footnote{{\it i.e.} its coefficients are real, we will see later the
combinatorial content of them.} $\Gamma(\overline{z})=\overline{\Gamma(z)}$.
From Weierstrass factorization \cite{Dieudonne} and Newton-Girard identity \cite{lascoux},
we have successively
\begin{eqnarray}\label{Weierstrass}
\frac1{\Gamma(z+1)}
=e^{\gamma z}\prod_{n\ge1}\Bigl(1+\frac{z}{n}\Bigr)e^{-\frac{z}{n}}
=\exp\Bigl(\gamma z-\sum_{k\ge2}\zeta(k)\frac{(-z)^k}{k}\Bigr).
\end{eqnarray}

Using the following functional equation and Euler's complement formula, {\it i.e.}
\begin{eqnarray*}
\Gamma(1+z)=z\Gamma(z)&\mbox{and}&\Gamma(z)\Gamma(1-z)=\frac{\pi}{\sin(z\pi)},
\end{eqnarray*}
and also introducing the {\it partial beta function} defined (for any $a,b,z\in\C$
such that $\Re a>0,\Re b>0,\abs{z}<1$) by
\begin{eqnarray}\label{beta}
\mathrm{B}(z;a,b):=\int_0^zdt\;t^{a-1}(1-t)^{b-1}
\end{eqnarray}
and then, classically, $\mathrm{B}(a,b):=\mathrm{B}(1;a,b)={\Gamma(a)\Gamma(b)}/{\Gamma(a+b)}$,
one has (for any $u,v\in\C$ such that $\abs{u}<1,\abs{v}<1$ and $\abs{u+v}<1$) the following expression 
\begin{eqnarray}\label{beta1}
\exp\biggl(-\sum_{n\ge2}{\zeta(n)}\frac{(u+v)^n-(u^n+v^n)}n\biggr)
&=&\frac{\Gamma(1-u)\Gamma(1-v)}{\Gamma(1-u-v)},\\
&=&\frac{\Gamma(u+v)}{\Gamma(u)\Gamma(v)}\pi\frac{\sin((u+v)\pi)}{\sin(u\pi)\sin(v\pi)}\\
&=&\frac{\pi}{{\mathrm{B}(u,v)}}(\cot(u\pi)+\cot(v\pi)).\label{beta2}
\end{eqnarray}
In particular, for $v=-u$ ($\abs{u}<1$), one gets
\begin{eqnarray*}
\exp\Bigl(-\sum_{k\ge1}{\zeta(2k)}\frac{u^{2k}}{k}\Bigr)
=\frac1{{\Gamma(1-u)\Gamma(1+u)}}=\frac{\sin(u\pi)}{u\pi}.
\end{eqnarray*}
Hence, taking the logarithms and considering Taylor expansions, one obtains
\begin{eqnarray}\label{Euler}
-\sum_{k\ge1}{\zeta(2k)}\frac{u^{2k}}{k}
&=&\log\Bigl(1+\sum_{n\ge1}\frac{(u\mathrm{i}\pi)^{2n}}{{\Gamma(2n+2)}}\Bigr)\\
&=&\sum_{k\ge1}(u\mathrm{i}\pi)^{2k}\sum_{l\ge1}\frac{(-1)^{l-1}}{l}
\sum_{n_1,\ldots,n_l\ge1\atop n_1+\ldots+n_l=k}\prod_{i=1}^l\frac{1}{{\Gamma(2n_i+2)}}.
\end{eqnarray}
One can deduce then the following expression\footnote{
Note that Euler gave another explicit formula using Bernoulli numbers.} for ${\zeta(2k)}$:
\begin{eqnarray}\label{paire}
\frac{{\zeta(2k)}}{\pi^{2k}}=k\sum_{l=1}^k\frac{(-1)^{k+l-1}}{l}
\sum_{n_1,\ldots,n_l\ge1\atop n_1+\ldots+n_l=k}\prod_{i=1}^l\frac{1}{{\Gamma(2n_i+2)}}\in\Q.
\end{eqnarray}

Now, more generally, for any $r\in\N_{\ge1}$ and $(s_1,\ldots,s_r)\in{\mathbb C}^r$,
let us consider the following {\it several variable zeta function}
\begin{eqnarray}
\zeta(s_1,\ldots,s_r):=\sum_{n_1>\ldots>n_r>0}n_1^{-s_1}\ldots n_r^{-s_r}
\end{eqnarray}
which converges for $(s_1,\ldots,s_r)$ in the open sub-domain of $\C^r,r\ge1,$ \cite{Goncharov,Zhao}
\begin{eqnarray*}
\calH_r:=\{(s_1,\ldots,s_r)\in\C^r\,\vert\,
\forall m=1,\ldots,r,\,\Re(s_1)+\ldots+\Re(s_m)>m\}.
\end{eqnarray*}
In the convergent cases, from a theorem by Abel, for $n\in\N,z\in\C,\abs{z}<1$,
its values can be obtained as the following limits
\begin{eqnarray}\label{Abel}
\zeta(s_1,\ldots,s_r)
=\Lim_{z\rightarrow1}\Li_{s_1,\ldots,s_r}(z)
=\Lim_{n\rightarrow+\infty}\H_{s_1,\ldots,s_r}(n),
\end{eqnarray}
where the following {\it polylogarithms} are well defined
\begin{eqnarray}
\Li_{s_1,\ldots,s_r}(z)&:=&\sum_{n_1>\ldots>n_r>0}\frac{z^{n_1}}{n_1^{s_1}\ldots n_r^{s_r}},\\
\frac{\Li_{s_1,\ldots,s_r}(z)}{1-z}&=&\sum_{n\ge0}\H_{s_1,\ldots,s_r}(n)z^n,
\end{eqnarray}
and so are the Taylor coefficients\footnote{These quantities are generalizations
of the harmonic numbers $H_n=1+2^{-1}\ldots+n^{-1}$ to which they boil down for $r=1,s_1=1$.
They are also truncations of the zeta values $\zeta(s_1,\ldots,s_r)$ at order $n+1$.}
here simply called {\it harmonic sums} 
\begin{eqnarray}
\H_{s_1,\ldots,s_r}:\N&\longrightarrow&\Q\mbox{({\it i.e. } an arithmetic function)},\\
n&\longmapsto&\H_{s_1,\ldots,s_r}(n)=\sum_{n\ge n_1>\ldots>n_r>0}n_1^{-s_1}\ldots n_r^{-s_r}.
\end{eqnarray}

On $\calH_r\cap\N^r$, the polyzetas can be represented by the following integral representation\footnote{
On $\calH_r$, $\log(a/b)$ is replaced by $\log(a)-\log(b)$.} over $]0,1[$ \cite{IMACS}
(here, one set $\lambda(z):={z}(1-z)^{-1},t_0=1$ and $u_{r+1}=1$):
\begin{eqnarray}
\zeta(s_1,\ldots,s_r)
&=&\int_0^1\omega_1(t_1)\frac{\log^{s_1-1}(t_0/t_1)}{\Gamma(s_1)}
\ldots\int_0^{t_{r-1}}\omega_1(t_r)\frac{\log^{s_r-1}(t_{r-1}/t_r)}{\Gamma(s_r)}\nonumber\\
&=&\prod_{i=1}^r\frac1{\Gamma(s_i)}
\int_{[0,1]^r}\prod_{j=1}^r\omega_0(u_j)\lambda(u_1\ldots u_j)\log^{s_j-1}(\frac1{u_j})\nonumber\\
&=&\prod_{i=1}^r\frac1{\Gamma(s_i)}\int_{\R_+^r}\prod_{j=1}^r\omega_0(u_j)u_j^{s_j}\lambda(e^{-(u_1\ldots u_j)}),\label{LambdaExpr}
\end{eqnarray}
with $\omega_0(z)=dz/z$ and $\omega_1(z)=dz/(1-z)$.

As for the Riemann zeta function in \eqref{zeta}, we observe that (\ref{LambdaExpr})
involves again the factors (and products) of eulerian Gamma function and also their
quotients (hence, eulerian Beta function). In the sequel, in continuation with \cite{CASC2018,Ngo2,CM},
we propose to study the ratios $\zeta(s_1,\ldots,s_r)/\pi^{s_1+\ldots+s_r}$ (and others),
an analogue of \eqref{paire}, which will be achieved as consequence of
regularizations, via the values of {\it entire} functions, of {\it divergent}
polyzetas and infinite sums of polyzetas (see Theorem \ref{Newton-Girard}
and Corollaries \ref{identity}, \ref{C1} in Section \ref{regularization})
for which a theorem by Abel (see \eqref{Abel}) could not help any more.
This achievement is justified thanks to the extensions of polylogarithms and harmonic sums
(see Theorems \ref{Indexation} and \ref{Indexation2} in Section \ref{Polylogarithms})
and thanks to the study of the independence of a family of eulerian functions
which can be viewed as generating series of zeta values (for $r\ge2$):
\begin{eqnarray}
\frac{1}{\Gamma_{y_r}(z+1)}
=\sum_{k\ge0}\zeta(\underbrace{{r,\ldots,r}}_{k{\tt times}})z^{kr}
=\exp\Bigl(-\sum_{k\ge1}\zeta(kr)\frac{(-z^r)^{k}}k\Bigr)
\end{eqnarray}
(see Propositions \ref{Weierstrass2}--\ref{Weierstrass3} and Theorem \ref{independence}
in Section \ref{eulerianfunctions}) via the combinatorial tools introduced in
Section \ref{combinatorialframeworks} (see Lemma \ref{ind_lin}, \ref{general}
in Section \ref{combinatorialframeworks}). Finally, identities among these
(convergent or divergent) generating series of zeta values are suitable
to obtain relations, at arbitrary weight, among polyzetas (see Examples
\ref{example1} and \ref{example2} in Section \ref{regularization}).

\section{Families of eulerian functions}
In all the sequel, $\C[\{f_i\}_{i\in I}]$ denotes the algebra generated by $\{f_i\}_{i\in I}$,
$\dext{\C}{(g_i)_{i\in I}}$ denotes the differential $\C$-algebra\footnote{{\it i.e.} the
$\C$-algebra generated by $g_i$ and their derivatives \cite{VdP}.}, generated by the family
$(g_i)_{i\in I}$ of the $\C$-commu\-tative differential ring $(\calA,\partial)$ ($1_{\calA}$
is its neutral element) and $\calC_0$ denotes a differential subring of $\calA$
($\partial\calC_0\subset\calC_0$) which is an integral domain containing the field of constants.
If the ring $\calA$ is without zero divisors then the fields of fractions $\mathrm{Frac}(\calC_0)$
and $\mathrm{Frac}(\A)$ are naturally differential fields and can be seen as the smallest ones
containing $\calC_0$ and $\calA$, respectively, satisfying
$\mathrm{Frac}(\calC_0)\subset\mathrm{Frac}(\calA)$.

\subsection{Words and formal power series}\label{combinatorialframeworks}

Let $\calX$ denote either the alphabets $X:=\{x_0,x_1\}$ or $Y:=\{y_k\}_{k\ge1}$,
equipped with a total ordering, and let $\calX^*$ denote the monoid freely generated by $\calX$
(its unit is denoted by $1_{\calX^*}$). The set of noncommutative polynomials (resp. series)
over $\calX$ with coefficients in a commutative ring $A$, containing $\Q$, is denoted by
$\ncp{A}{\calX}$ (resp. $\ncs{A}{\calX}$) \cite{berstel}.
The algebraic closure of\footnote{In general, $\widehat{A.{\calX}}$ is the module of
homogeneous series $S\in \ncs{A}{\calX}$ of degree one.} $\widehat{A.{\calX}}$ by the
rational operations\footnote{Here ${\tt conc}$ stand for the Cauchy product (concatenation)
and $\Delta_{{\tt conc}}$ is its co-product.

For any $S\in\ncs{A}{\calX}$ such that $\scal{S}{1_{{\calX}^*}}=0$, the Kleene star of $S$ is defined by
$S^*:=(1-S)^{-1}=1+S+S^2+\ldots$.}
$\{{\tt conc},+,*\}$ within $\ncs{A}{\calX}$ is denoted by $\ncs{A^{\mathrm{rat}}}{\calX}$ \cite{berstel}.
We will also consider the following Hopf algebras and, in the case of $A=\mathbf{k}$ being a field, their
Sweedler's dual\footnote{Here, $\shuffle$ (resp. $\stuffle$) stand for the shuffle (resp. stuffle) product
and $\Delta_{\shuffle}$ (resp. $\Delta_{\stuffle}$) is its co-product (see \cite{reutenauer} or \cite{Ngo2}).

The antipode of the first one is given by $a(w)=(-1)^{|w|}\widetilde{w}$, the antipode of
the second one exists because the bialgebra is graded by weight, but is more complicated.} \cite{PVNC,CM}
\begin{eqnarray}
(\ncp{A}{{\calX}},{\tt conc},\Delta_{\shuffle},1_{{\calX}^*},\epsilon)&\mbox{and}&
(\ncp{A}{Y},{\tt conc},\Delta_{\stuffle},1_{Y^*},\epsilon),\\
(\ncs{\mathbf{k}^{\mathrm{rat}}}{\calX},\shuffle,1_{{\calX}^*},\Delta_{{\tt conc}},\epsilon)&\mbox{and}&
(\ncs{\mathbf{k}^{\mathrm{rat}}}{Y},\stuffle,1_{Y^*},\Delta_{{\tt conc}},\epsilon).
\end{eqnarray}
In particular, using the set of Lyndon words, denoted by $\Lyn\calX$, one constructs the basis
$\{P_l\}_{l\in\Lyn\calX}$, for $\ncp{\calL ie_{A}}{\calX}$, generating the
PBW-Lyndon basis $\{P_w\}_{w\in\calX^*}$ for $(\ncp{A}{{\calX}},{\tt conc},1_{{\calX}^*})$
and then the graded dual basis $\{S_w\}_{w\in\calX^*}$ containing the pure transcendence
basis $\{S_l\}_{l\in\Lyn\calX}$ for the shuffle algebra $(\ncs{A}{\calX},\shuffle,1_{\calX^*})$.
Similarly, the basis $\{\Pi_l\}_{l\in\Lyn Y}$ generating the PBW-Lyndon basis $\{\Pi_w\}_{w\in Y^*}$
for $(\ncp{A}{Y},{\tt conc},1_{Y^*})$ and then the graded dual basis $\{\Sigma_w\}_{w\in Y^*}$
containing the pure transcendence basis $\{\Sigma_l\}_{l\in\Lyn Y}$ for the stuffle algebra
$(\ncs{A}{Y},\stuffle,1_{Y^*})$.

\begin{lemma}\label{general}
\begin{enumerate}
\item The algebras $(\C[\{x^*\}_{x\in\calX}],\shuffle,1_{\calX^*})$ and
$(\ncp{\C}{\calX},\shuffle,1_{\calX^*})$ are algebraically disjoint over $\C$ and
\begin{eqnarray*}
(\ncp{\C[\{x^*\}_{x\in\calX}]}{\calX},\shuffle,1_{\calX^*})
&\cong&(\C[\{x^*\}_{x\in\calX}][\Lyn\calX],\shuffle,1_{\calX^*})\\
&\cong&(\C[\{x^*,l\}_{x\in\calX,l\in\Lyn\calX}],\shuffle,1_{\calX^*})
\end{eqnarray*}
which is generated by the transcendent basis $\{x^*,l\}_{x\in\calX,l\in\Lyn\calX}$ over $\C$. 
\item Let $K:=\C[\{f(x^*)\}_{x\in\calX}]$and $F:=\C[\{f(l)\}_{l\in\Lyn\calX}]$.

Let $f$ be the shuffle morphism
$(\ncp{\C[\{x^*\}_{x\in\calX}]}{\calX},\shuffle,1_{\calX^*})\longrightarrow(\calA,\times,1_{\calA})$.

Then the following assertions are equivalent
\begin{enumerate}
\item\label{CN} The morphism $f$ is injective.
\item\label{CS} The algebras $K$ and $F$, satisfying $K\cap F=\C.1_{\calA}$, are generated by the
transcendent bases $\{f(x^*)\}_{x\in\calX}$ and $\{f(l)\}_{l\in\Lyn\calX}$, respectively, over $\C$.
\end{enumerate}

Hence, if \ref{CN}, or \ref{CS}, holds then $F,K$ are algebraically disjoint over $\C$ and
\begin{eqnarray*}
\C[\{f(x^*)\}_{x\in\calX}][\{f(l)\}_{l\in\Lyn\calX}]\cong\C[\{f(x^*),f(l)\}_{x\in\calX,l\in\Lyn\calX}]
\end{eqnarray*}
which is generated by the transcendent basis $\{f(x^*),f(l)\}_{x\in\calX,l\in\Lyn\calX}$ over $\C$.
\end{enumerate}  
\end{lemma}

\begin{proof}
\begin{enumerate}
\item Recall that the algebras $(\C[\{x^*\}_{x\in\calX}],\shuffle,1_{\calX^*})$
and $(\ncp{\C}{\calX},\shuffle,1_{\calX^*})$ are generated, respectively, by the
transcendent bases $\{x^*\}_{x\in\calX}$ \cite{PVNC} and $\Lyn\calX$ \cite{reutenauer}.
Moreover, $\{x^*\}_{x\in\calX}$ is also algebraically independent over $\ncp{\C}{\calX}$
\cite{PVNC} and then $\C[\{x^*\}_{x\in\calX}]\cap\ncp{\C}{\calX}=\C.1_{\calX^*}$.
It follows the then expected results.
\item Straightforward.
\end{enumerate}
\end{proof}

Now, for any $r\ge1$, let us consider the following differential form
\begin{eqnarray}\label{omega}
\omega_r(z)=u_{y_r}(z)dz&\mbox{with}&u_{y_r}\in\calC_0\subset\calA.
\end{eqnarray}
Let us also consider the following noncommutative differential equation (see \cite{PVNC})
\begin{eqnarray}\label{NCDE_abs}
{\bf d}S=MS;&\scal{S}{1_{\calX^*}}=1_{\calA},
&\mbox{where }M=\sum_{x\in{\calX}}u_x x\in\widehat{\calC_0{\calX}},
\end{eqnarray}
where ${\bf d}$ is the differential operator on $\ncs{\calA}{{\calX}}$ extending $\partial$ as follows: 
\begin{eqnarray}
\forall S=\sum_{w\in\calX^*}\scal{S}{w}w\in\ncs{\calA}{\calX},&&{\bf d}S=\sum_{w\in\calX^*}(\partial\scal{S}{w})w.
\end{eqnarray}

In order to prove Proposition \ref{plein}, Theorems \ref{independence} and \ref{Indexation} below,
we use  the following lemma, a particular case of a general localization result to be proved in a
forthcoming paper \cite{PVNC}.

\begin{lemma}\label{ind_lin}
Suppose that the $\C$-commutative ring $\calA$ is without zero divisors and equipped with
a differential operator $\partial$ such that $\C=\ker\partial$.

Let $S\in\ncs{\calA}{{\calX}}$ be a group-like solution of \eqref{NCDE_abs}, in the following form
\begin{eqnarray*}
S=1_{\calX^*}+\sum_{w\in{\calX}^*\calX}\scal{S}{w}w
=1_{\calX^*}+\sum_{w\in{\calX}^*\calX}\scal{S}{S_w}P_w
=\prod_{l\in\Lyn{\calX}}^{\searrow}e^{\scal{S}{S_l}P_l}.
\end{eqnarray*}
Then
\begin{enumerate}
\item If $H\in\ncs{\calA}{{\calX}}$ is another group-like solution of \eqref{NCDE_abs}
then there exists $C\in\ncs{\calL ie_{\calA}}{{\calX}}$ such that $S=He^C$ (and conversely).
\item The following assertions are equivalent
\begin{enumerate}
\item\label{item1} $\{\scal{S}{w}\}_{w\in{\calX}^*}$ is $\calC_0$-linearly independent,
\item\label{item2} $\{\scal{S}{l}\}_{l\in\Lyn{\calX}}$ is $\calC_0$-algebraically independent,
\item\label{item3} $\{\scal{S}{x}\}_{x\in{\calX}}$ is $\calC_0$-algebraically independent,
\item\label{item4} $\{\scal{S}{x}\}_{x\in{\calX}\cup\{1_{{\calX}^*}\}}$ is $\calC_0$-linearly independent,
\item\label{item5} The family $\{u_x\}_{x\in{\calX}}$ is such that, for $f\in\mathrm{Frac}(\calC_0)$
and $(c_x)_{x\in\calX}\in\C^{(\calX)}$,
\begin{eqnarray*}
\sum_{x\in{\calX}}c_xu_x=\partial f&\Rightarrow&(\forall x\in{\calX})(c_x=0).
\end{eqnarray*}
\item\label{item6} The family $(u_x)_{x\in{\calX}}$ is free over $\C$ and
$\partial\mathrm{Frac}(\calC_0)\cap\span_{\C}\{u_x\}_{x\in{\calX}}=\{0\}$.
\end{enumerate}
\end{enumerate}
\end{lemma}

\begin{proof}[Sketch]
The first item has been treated in \cite{orlando}. The second is a group-like
version of the abstract form of Theorem 1 of \cite{Linz}. It goes as follows
\begin{enumerate}[\textbullet]
\item due to the fact that $\calA$ is without zero divisors, we have the following embeddings
$\calC_0\subset\mathrm{Frac}(\calC_0)\subset\mathrm{Frac}(\A)$, $\mathrm{Frac}(\A)$
is a differential field, and its derivation can still be denoted by $\partial$
as it induces the previous one on $\calA$,
\item the same holds for $\ncs{\calA}{{\calX}}\subset\ncs{\mathrm{Frac}(\A)}{\calX}$ and ${\bf d}$
\item therefore, equation \eqref{NCDE_abs} can be transported in $\ncs{\mathrm{Frac}(\A)}{\calX}$
and $M$ satisfies the same condition as previously. 
\item Equivalence between \ref{item1}-\ref{item4} comes from the fact that $\calC_0$ is without zero divisors
and then, by denominator chasing, linear independances w.r.t $\calC_0$ and $\mathrm{Frac}(\calC_0)$ are equivalent.
In particular, supposing condition \ref{item4}, the family $\{\scal{S}{x}\}_{x\in{\calX}\cup\{1_{{\calX}^*}\}}$
(basic triangle) is  $\mathrm{Frac}(\calC_0)$-linearly independent which imply, by the Theorem 1 of \cite{Linz},
condition \ref{item5},
\item still by Theorem 1 of \cite{Linz}, \ref{item5} is equivalent to \ref{item6}, implying that
$\{\scal{S}{w}\}_{w\in{\calX}^*}$ is $\mathrm{Frac}(\calC_0)$-linearly independent which induces
$\calC_0$-linear independence ({\it i.e.} \ref{item1}).
\end{enumerate}  
\end{proof}

Now, let $\calA=\calH(\Omega)$, the ring of holomorphic functions on a simply connected domain
$\Omega\subset\C$ ($1_{\calH(\Omega)}$ is its neutral element). With the notations in \eqref{omega}
and for any path $z_0\path z$ in $\Omega$, let
$\alpha_{z_0}^z:(\ncs{\C^{\mathrm{rat}}}{\calX},\shuffle,1_{\calX^*})\longrightarrow(\calH(\Omega),\times,1_{\calA})$
be the morphism defined, for any $x_{i_1}\ldots x_{i_k})\in{\calX^*}$, by\cite{PVNC}
\begin{eqnarray}\label{notation}
\alpha_{z_0}^z(x_{i_1}\ldots x_{i_k})=\int_{z_0}^z\omega_{{i_1}}(z_1)\ldots\int_{z_0}^{z_{k-1}}\omega_{{i_k}}(z_k)
&\mbox{and}&\alpha_{z_0}^z(1_{\calX^*})=1_{\calH(\Omega)},
\end{eqnarray}
satisfying $\alpha_{z_0}^z(u\shuffle v)=\alpha_{z_0}^z(u)\alpha_{z_0}^z(v)$, for $u,v\in\calX^*$.
By a Ree's theorem \cite{reutenauer}, the Chen series of $\{\omega_r\}_{r\ge1}$ and
along the path $z_0\path z$ in $\Omega$ is group-like:
\begin{eqnarray}\label{Chen}
C_{z_0\path z}=\sum_{w\in\calX^*}\alpha_{z_0}^z(w)w
=\prod_{l\in\Lyn\calX}^{\searrow}e^{\alpha_{z_0}^z(S_l)P_l}\in\ncs{\calH(\Omega)}{\calX}.
\end{eqnarray}
Since
$\partial\alpha_{z_0}^z(x_{i_1}\ldots x_{i_k})=u_{i_1}(z)\alpha_{z_0}^z(x_{i_2}\ldots x_{i_k})$
then $C_{z_0\path z}$ is a solution of \eqref{NCDE_abs}.

\begin{remark}\label{evaluation}
For any $w\in\calX\calX^*$, the value of $\alpha_{z_0}^z(w)$ depends on $\{\omega_i\}_{i\ge1}$, or equivalently
on $\{u_x\}_{x\in\calX}$ and if $f_x(z)=\alpha_{z_0}^z(x)$ then, for any $n\ge0$, one has \cite{IMACS}
\begin{eqnarray*}
\alpha_{z_0}^z(x^n)=\alpha_{z_0}^z(x^{\shuffle n}/n!)=f_x^n(z)n!&\mbox{and then}&F_x(z):=\alpha_{z_0}^z(x^*)=e^{f_x(z)}.
\end{eqnarray*}
\end{remark}

With data in \eqref{omega} and shuffle morphism in \eqref{notation},
we will illustrate a bijection, between
$(\ncp{\C}{\calX}\shuffle\C[\{x^*\}_{x\in\calX}],\shuffle,1_{\calX^*})$,
the subalgebra of noncommutative rational series and a subalgebra of $\calH(\Omega)$
containing the eulerian functions bellow.

\subsection{Families of eulerian functions}\label{eulerianfunctions}

\begin{definition}\label{f+expf}
For any $z\in\C$ such that $\abs{z}<1$, we put
\begin{eqnarray*}
\ell_1(z):=\gamma z-\sum_{k\ge2}\zeta(k)\frac{(-z)^k}k&\mbox{and for}&
r\ge2,\ell_r(z):=-\sum_{k\ge1}\zeta(kr)\frac{(-z^r)^k}k.
\end{eqnarray*}
For any $k\ge1$, let $\Gamma_{y_k}(1+z):=e^{-\ell_k(z)}$ and
$\mathrm{B}_{y_k}(a,b):=\dfrac{\Gamma_{y_k}(a)\Gamma_{y_k}(b)}{\Gamma_{y_k}(a+b)}$.
\end{definition}

\begin{remark}\label{Gamma}
\begin{enumerate}
\item $(\ell_r)_{r\ge1}$ is triangular\footnote{$(g_i)_{i\ge1}$ is said to be {\it triangular}
if the valuation of $g_i,\varpi(g_i),$ equals $i\ge1$. It is easy to check that such a family is
$\C$-linearly free and that is also the case of families such that $(g_i-g(0))_{i\ge1}$ is triangular.}.
So is $(e^{\ell_r}-e^{\ell_r(0)})_{r\ge1}$.

\item For any $z\in\Omega=\C,\abs{z}<1$ and $k\ge1$, using Remark \ref{evaluation}, one has
\begin{eqnarray*}
\begin{array}{|c||c|c|}
\hline
u_{y_k}&\alpha_0^z(y_k)&\alpha_0^z(y_k^*)\\
\hline
\hline
1_{\calH(\Omega)}&z&e^z\\
\hline
(1-z)^{-1}&-\log(1-z)&(1-z)^{-1}\\
\hline
\partial\ell_k&\ell_k(z)&e^{\ell_k(z)}=\Gamma_{y_k}^{-1}(1+z)\\
\hline
e^{\ell_k}\partial\ell_k&e^{\ell_k(z)}=\Gamma_{y_k}^{-1}(1+z)&e^{e^{\ell_k(z)}}\\
\hline
\end{array}
\end{eqnarray*}
\item The function $\ell_1$ is already considered by Legendre
for studying the eulerian Beta and Gamma functions \cite{legendre},
denoted here, repectively, by $\mathrm{B}_{y_1}$ and $\Gamma_{y_1}$.
\item For any $r\ge1$, one has $\partial\ell_r=e^{-\ell_r}\partial e^{\ell_r}$.
\item For any $n\ge0$, one puts classically $\Psi_{n}:=\partial^n\log\Gamma$.
\item Some of these functions cease (unlike $\Gamma$) to be hypertranscendental.
For example\footnote{Indeed, we use the fact that $\Gamma_{y_2}^{-1}(1+x)
=\sin(\mathrm{i}\pi x)/\mathrm{i}\pi x$ (see Example \ref{example1} bellow).}
$y(x)=\Gamma_{y_2}^{-1}(1+x)$ is a solution of $(1-\pi^2x^2)y^2 +2xy\dot y+x^2\dot y^2=1$.
\end{enumerate}
\end{remark}

Now, for any $r\ge1$, let $G_r$ (resp. $\calG_r$) denote the set (resp. group) of solutions,
$\{\xi_0,\ldots,\xi_{r-1}\}$, of the equation $z^r=(-1)^{r-1}$ (resp. $z^r=1$).
For $r,q\ge1$, we will need also a system $\mathbb{X}$ of representatives of
$\calG_{qr}/\calG_r$, {\it i.e.} $\mathbb{X}\subset \calG_{qr}$ such that
\begin{eqnarray*}
\calG_{qr}=\biguplus_{\tau\in\mathbb{X}}\tau\calG_r.
\end{eqnarray*}
It can also be assumed that $1\in\mathbb{X}$ as with $\mathbb{X}=\{e^{2\mathrm{i}k\pi/qr}\}_{0\le k\le q-1}$.

\begin{remark}\label{parity}
If $r$ is odd then $z^r=(-1)^{r-1}=1$ and $G_r=\calG_r$ as being a group otherwise
$G_r=\xi\calG_r$ as being an orbit, where $\xi$ satisfies $\xi^r=-1$
(this is equivalent to $\xi\in\calG_{2r}$ and $\xi\notin\calG_r$).
\end{remark}

\begin{proposition}\label{Weierstrass2}
\begin{enumerate}
\item\label{generator} For $r\ge1,\chi\in\calG_r$ and $z\in\C,|z|<1$, the functions $\ell_r$ and $e^{\ell_r}$ have the symmetry,
$\ell_r(z)=\ell_r(\chi z)$ and $e^{\ell_r(z)}=e^{\ell_r(\chi z)}$. In particular, for $r$ even, as $-1\in\calG_r$, these functions are even.
\item For $|z|<1$, we have
\begin{eqnarray*}
\ell_r(z)=-\sum_{\chi\in G_r}\log(\Gamma(1+\chi z))&\mbox{and}&
e^{\ell_r(z)}=\prod_{\chi\in G_r}e^{\gamma\chi z}\prod_{n\ge1}\Bigl(1+\frac{\chi z}{n}\Bigr)e^{-{\chi z}/n}.
\end{eqnarray*}
\item For any odd $r\ge2$,
\begin{eqnarray*}
\Gamma_{y_r}^{-1}(1+z)=e^{\ell_r(z)}=\Gamma^{-1}(1+z)\prod_{\chi\in G_r\setminus{\{1\}}}e^{\ell_1(\chi z)}.
\end{eqnarray*}
\item In general, for any odd or even $r\ge2$,
\begin{eqnarray*}
e^{\ell_r(z)}=\prod_{\chi\in G_r}e^{\ell_1(\chi z)}=\prod_{n\ge1}\Bigl(1+\frac{z^r}{n^r}\Bigr).
\end{eqnarray*}
\end{enumerate}
\end{proposition}

\begin{proof}
The results are known for $r=1$ ({\it i.e.} for $\Gamma^{-1}$). For $r\ge2$, we get
\begin{enumerate}
\item By Definition \ref{f+expf}, with $\chi\in \calG_r$, we get
\begin{eqnarray*}
\ell_r(\chi z)=-\sum_{k\ge1}\zeta(kr)\frac{(-\chi^rz^r)^k}k
=-\sum_{k\ge1}\zeta(kr)\frac{(-z^r)^k}k=\ell_r(z),
\end{eqnarray*}
thanks to the fact that, for any $\chi\in\calG_r$, one has $\chi^r=1$.
In particular, if $r$ is even then $\ell_r(z)=\ell_r(-z)$, {\it i.e.} $\ell_r$ is even.
\item If $r$ is odd, as $G_r=\calG_r$ and, applying the symmetrization principle\footnote{
Within the same disk of convergence as $f$, one has, $f(z)=\sum_{n\ge1}a_nz^n$ and 
$\sum_{\chi\in\calG_r}f(\chi z)=r\sum_{k\ge1}a_{rk}z^{rk}$.}, we get 
\begin{eqnarray*}
-\sum_{\chi\in G_r}\ell_1(\chi z)=&-\displaystyle\sum_{\chi\in\calG_r}\ell_1(\chi z)\\
=&r\displaystyle\sum_{k\ge1}\zeta(kr)\frac{(-z)^{kr}}{kr}&=\sum_{k\ge1}\zeta(kr)\frac{(-z^r)^k}{k}.
\end{eqnarray*}
The last term being due to, precisely, $r$ is odd. If $r$ is even, we have the
orbit $G_r=\xi\calG_r$ (still with $\xi^r=-1$) and then, by the same principle,
\begin{eqnarray*}
-\sum_{\chi\in\calG_r}\ell_1(\chi\xi z)=&r\displaystyle\sum_{k\ge1}\zeta(kr)\frac{(-\xi z)^{kr}}{kr}\cr
=&\displaystyle\sum_{k\ge1}\zeta(kr)\frac{\bigl((-\xi z)^{r}\bigr)^k}{k}&=\sum_{k\ge1}\zeta(kr)\frac{\bigl(-z^{r}\bigr)^k}{k}.
\end{eqnarray*}
\item Straightforward.
\item Due to the fact that the external product is finite, we get
\begin{eqnarray*}
e^{\ell_r(z)}
=\overbrace{\Bigl(\prod_{\chi\in\calG_r}e^{\gamma\chi z}\Bigr)}^{=1}
\prod_{n\ge1\atop\chi\in\calG_r}\Bigl(1+\frac{\chi z}{n}\Bigr)e^{-{\chi z}/n}
=\overbrace{\Bigl(\prod_{n\ge1\atop\chi\in\calG_r}e^{-\frac{\chi z}n}\Bigr)}^{=1}
\prod_{n\ge1\atop\chi\in\calG_r}\Bigl(1+\frac{\chi z}{n}\Bigr).
\end{eqnarray*}
Using the elementary symmetric functions of $G_r$, we get the expected result.
\end{enumerate}
\end{proof}

\begin{proposition}\label{plein}
Let $L:=\span_\C\{\ell_r\}_{r\ge1}$ and $E:=\span_{\C}\{e^{\ell_r}\}_{r\ge1}$.
Let $\C[L]$ and $\C[L]$ be their respective algebra. One has
\begin{enumerate}
\item\label{free} The family $(\ell_r)_{r\ge1}$ is $\C$-linearly free and free from $1_{\calH(\Omega)}$.
\item\label{free2} The family $(\ell_r)_{r\ge1}$ and $(e^{\ell_r})_{r\ge1}$ is $\C$-linearly free and free from $1_{\calH(\Omega)}$.
\item The families $(\ell_r)_{r\ge1}$ and $(e^{\ell_r})_{r\ge1}$ are $\C$-algebraically independent.
\item For any $r\ge1$, one has
\begin{enumerate}
\item The functions $\ell_r$ and $e^{\ell_r}$ are $\C$-algebraically independent.
\item The function $\ell_r$ is holomorphic on the open unit disc, $D_{<1}$,
\item The function $e^{\ell_r}$ (resp. $e^{-\ell_r}$) is entire (resp. meromorphic),
and admits a countable set of isolated zeroes (resp. poles) on the complex plane
which is expressed as $\biguplus_{\chi\in G_r}\chi\Z_{\le-1}$.
\end{enumerate}
\item One has $E\cap L=\{0\}$ and, more generally, $\C[E]\cap\C[L]=\C.1_{\calH(\Omega)}$.
\end{enumerate}
\end{proposition}

\begin{proof}
\begin{enumerate}
\item Suppose that there is $(a_r)_{r\ge1}\in\C^{(\N)}$ such that
\begin{eqnarray*}
\sum_{r\ge1}a_r\ell_r(z)=a_1\gamma z-\sum_{k\ge2}a_1\zeta(k)\frac{(-1)^k}{k}z^k
-\sum_{r\ge2}\sum_{k\ge1}a_r\zeta(kr)\frac{(-1)^k}{k}z^{rk}=0,
\end{eqnarray*}
in which, since $\gamma\neq0$ then $a_1=0$. It follows that
\begin{eqnarray*}
\sum_{r\ge2}\sum_{k\ge1}a_r\zeta(kr)\frac{(-1)^k}{k}z^{rk}=0
\end{eqnarray*}
in which $\scal{\mathrm{LS}}{z^2}=a_2\zeta(2)/2$. Since $\zeta(2)\neq9$ then $a_2=0$.
It also follows that
\begin{eqnarray*}
\sum_{r\ge3}\sum_{k\ge1}a_r\zeta(kr)\frac{(-1)^k}{k}z^{rk}=0.
\end{eqnarray*}
In similar way, one proves that $a_r=0$, for $r\in\mathbb{N}^+$. Hence, $(\ell_r)_{r\geq 1}$ is $\C$-free.

\item Suppose that there is $(b_i)_{i\ge1}\in\C^{(\N)}$ such that
\begin{eqnarray*}
\sum_{i\ge1}a_ie^{\ell_i}=0&\mbox{and then}&\sum_{i\ge1}a_i\dot\ell_i=0
\end{eqnarray*}
(taking the logarithmic derivative). By integration, one deduces then $(\ell_r)_{r\geq 1}$
is $\C$-linearly dependent contradicting with the item \ref{free}.
It remains that $(e^{f_i})_{i\in I}$ is $\C$-free.

\item Using Chen series of $\{\omega_r\}_{r\ge1}$ defined, as in Remark \ref{Gamma}, by $u_{x_r}=e^{\ell_r}\partial\ell_r$
(resp. $u_{x_r}=\partial\ell_r$), via items \ref{item1} or \ref{item2} of Lemma \ref{ind_lin}, $\{e^{\ell_r}\}_{r\ge1}$
(resp. $\{\ell_r\}_{r\ge1}$) is the $\C$-algebraically independent.

\item \begin{enumerate}
\item Since $\ell_r(0)=0,\partial e^{\ell_r}=e^{\ell_r}\partial\ell_r$
then $\ell_r$ and $e^{\ell_r}$ are $\C$-algebraically independent.

\item One has $e^{\ell_1(z)}=\Gamma^{-1}(1+z)$ which proves the claim for $r=1$.
For $r\ge2$, note that $1\leq\zeta(r)\le\zeta(2)$ which implies that the radius of convergence
of the exponent is $1$ and means that $\ell_r$ is holomorphic on the open unit disc.
This proves the claim.

\item $e^{\ell_r(z)}=\Gamma_{y_r}^{-1}(1+z)$ (resp. $e^{-\ell_r(z)}=\Gamma_{y_r}(1+z)$) is
entire (resp. meromorphic) as finite product of entire (resp. meromorphic) functions and,
by Proposition \ref{Weierstrass2}, Weierstrass factorization yields zeroes (resp. poles).
\end{enumerate}

\item Let $f\in E\cap L$ and then there is $\{c_y\}_{y\in Y}$ and $\{d_y\}_{y\in Y}\in\C^{(Y)}$ such that
\begin{eqnarray*}
f=\sum_{r\ge1}c_{y_r}\ell_r=\sum_{r\ge1}d_{y_r}e^{\ell_r}
\end{eqnarray*}
If $f\neq0$ then $\ell_{r_0},e^{\ell_{r_0}}$ could be linearly dependent, for some $r_0\ge1$, contradicting
with item \ref{free}. Hence, $E\cap L=\{0\}$.

$\C[E]$ (resp. $\C[F]$) is generated freely by $(e^{\ell_r})_{r\ge1}$ (resp. $(\ell_r)_{r\ge1}$) which
are entire (resp. holomorphic on $D_{<1}$) functions. Moreover, any $\C[E]\ni f\neq c1_{\Omega}$
($c\in\C$) is entire and then $f\notin\C[L]$ (and conversely). It follows the expected result.
\end{enumerate}
\end{proof}

By Lemma \ref{general}, Proposition \ref{plein} and Remark \eqref{Gamma}, one deduces then
\begin{corollary}
The map $\alpha_0^z:(\ncp{\C}{Y},\shuffle,1_{Y^*})\longrightarrow(\mathrm{span}_{\C}\{\alpha_0^z(w)\}_{w\in Y^*},\times,1_{\calH(\Omega)})$
is injective, for the inputs $\{\partial\ell_r\}_{r\ge1}$ or $\{e^{\ell_r}\partial\ell_r\}_{r\ge1}$, and then
$\{\alpha_0^z(w)\}_{w\in Y^*}$ (resp. $\{\alpha_0^z(l)\}_{l\in\Lyn Y}$ is linearly (resp. algebraically) independent over $\C$.
\end{corollary}

From now on the countable set of isolated zeros (resp. poles) of the entire (resp. meromorphic)
function $e^{\ell_r}$ (resp. $e^{-\ell_r}$) is denoted by $\calO(e^{\ell_r})$. We have 
\begin{eqnarray}\label{zeroes}
\calO(e^{\ell_r})=\biguplus_{\chi\in G_r}\chi\Z_{\le-1}.
\end{eqnarray}

\begin{example}
One has
\begin{eqnarray*}
 \calO(e^{\ell_1})&=&\Z_{\le-1},\cr
 \calO(e^{\ell_2})&=&-\mathrm{i}\Z_{\le-1}\uplus\mathrm{i}\Z_{\le-1}=\mathrm{i}\Z_{\neq0},\cr
\calO(e^{\ell_3}) &=& \Z_{\le-1}\uplus\mathrm{j}\Z_{\le-1}\uplus\mathrm{j}^2\Z_{\le-1},\cr
\calO(e^{\ell_4}) &=& (1+\mathrm{i})/\sqrt{2}\Z_{\neq0}\uplus(1-\mathrm{i})/\sqrt{2}\Z_{\neq0}.
\end{eqnarray*}
\end{example}

\begin{proposition}\label{Weierstrass3}
Let $\mathbb{X}$ denote any system of representatives of $\calG_{qr}/\calG_r$.
\begin{enumerate}
\item For any $r\ge1$ and odd $q\ge1$, one has, for $\abs{z}<1$,
\begin{eqnarray*}
e^{\ell_{qr}(z)}=\prod_{\chi\in\mathbb{X}}e^{\ell_{r}(\chi z)},&\mbox{or equivalently,}&
\Gamma^{-1}_{y_{qr}}(1+z)=\prod_{\chi\in\mathbb{X}}\Gamma^{-1}_{y_r}(1+\chi z).
\end{eqnarray*}
\item $e^{\ell_r}$ divides $e^{\ell_{qr}}$ if and only if $q$ is odd.
\item The full symmetry group of $e^{\ell_r}$ for the representation $s*f[z]=f(sz)$ is $\calG_r$.
\end{enumerate}
\end{proposition}

\begin{proof} 
\begin{enumerate}
\item Let $\xi$ be any root of $z^r=(-1)^{r-1}$, one remarks that, in all cases ($r$ be odd or even), we have
\begin{eqnarray*}
G_r=\xi\calG_r,&G_{qr}=\xi\calG_{qr},&\calG_{qr}=\biguplus_{\chi\in\mathbb{X}}\chi\calG_r.
\end{eqnarray*}
Then, by Proposition \ref{Weierstrass2}, we have 
\begin{eqnarray*}
\ell_{qr}(z)
=&\displaystyle\sum_{\chi\in G_{qr}}\ell_1(\chi z)\\
=&\displaystyle\sum_{\rho_1\in\calG_{qr}}\ell_1(\xi\rho_1 z)&=\sum_{\chi\in\mathbb{X},\rho_2\in\calG_r}\ell_1(\xi\rho_2\chi z)\\
=&\displaystyle\sum_{\chi\in\mathbb{X}\atop\rho_2\in\calG_r}\ell_1(\xi\rho_2(\chi z))&=\sum_{\chi\in\mathbb{X}}\ell_r(\chi z)\\
=&\ell_r(z)+\displaystyle\sum_{\chi\in\mathbb{X}\setminus\{1\}}\ell_r(\chi z).
\end{eqnarray*}
Last equality assumes that $1\in\mathbb{X}$. Taking exponentials, we get
\begin{eqnarray}\label{OrbFact}
e^{\ell_{qr}(z)}=\prod_{\chi\in\mathbb{X}}e^{\ell_r(\chi z)}
=e^{\ell_r(z)}\prod_{\chi\in\mathbb{X}\setminus\{1\}}e^{\ell_r(\chi z)}.
\end{eqnarray}
Again, first equality is general and the last assumes that $1\in\mathbb{X}$.

\item The fact that $e^{\ell_r}$ divides $e^{\ell_{qr}}$ if $q$ is odd comes from the 
factorization \eqref{OrbFact}.
Now, when $q$ even, it suffices to remark, from \eqref{zeroes}, that the opposite of any solution
of $z^r=-1$ is a zero of\footnote{More precisely, denoting $\mathbb{U}$ the unit circle, one has
$\calO(e^{\ell_r})\cap\mathbb{U}=-G_r\not=\emptyset$.} $e^{\ell_r}$ and 
$\calO(e^{\ell_{qr}})\cap\mathbb{U}=-G_{qr}$. But when $q$ is even one has
$-G_{r}\cap-G_{qr}=\emptyset$. Hence, in this case, $e^{\ell_r}$ cannot divide $e^{\ell_{qr}}$.

\item Let $\calG$ denote the symmetry group of $e^{\ell_r}$ and remark that the distance of $\calO(e^{\ell_r})$
to zero is $1$. Hence, as $\calO(e^{\ell_r}(s.z))=s^{-1}\calO(e^{\ell_r})$, we must have $\calG\subset\mathbb{U}$.
Then, by Remark \ref{parity}, as $\calO(e^{\ell_r})\cap\mathbb{U}=G_r$, we must have $\calG\subset\calG_r$,
the reverse inclusion is exactly the first point of Proposition \ref{Weierstrass2}.
\end{enumerate}
\end{proof}

\begin{example}
\begin{enumerate}
\item For $r=1,q=2,\mathbb{X}=\{1,-1\}$, one has the Euler's complement like formula, \textit{ie.}
$\Gamma_{y_2}(1+\mathrm{i}z)=\Gamma_{y_1}(1+z)\Gamma_{y_1}(1-z)={z\pi}/{\sin(z\pi)}$.
Changing $z\mapsto-\mathrm{i}z$, one also has
$\Gamma_{y_2}(1+z)=\Gamma_{y_1}(1+\mathrm{i}z)\Gamma_{y_1}(1-\mathrm{i}z)$.

\item For $r=2,q=3,\mathbb{X}=\{1,\mathrm{j},\mathrm{j}^2\}$, one has
$\Gamma_{y_{6}}(1+z)=\Gamma_{y_2}(1+z)\Gamma_{y_2}(1+\mathrm{j}z)\Gamma_{y_2}(1+\mathrm{j}^2z)$.
\end{enumerate}
\end{example}

\goodbreak

With the notations of Proposition \ref{plein}, the algebra $\C[L]$ (resp. $\C[E]$) is generated
freely by $(\ell_r)_{r\ge1}$ (resp. $(e^{\ell_r})_{r\ge1}$) which are holomorphic on $D_{<1}$
(resp. entire) functions. Moreover,
\begin{eqnarray}\label{cap1}
E\cap L=\{0\},&\mbox{and more generally},&\C[E]\cap\C[L]=\C.1_{\calH(\Omega)}.
\end{eqnarray}
We are in a position to consider the following differential subalgebras of $(\calH(\Omega),\partial)$: 
\begin{eqnarray}
\calL:=\dext{\C}{(\ell_r^{\pm1})_{r\ge1}}&\mbox{and}&\calE:=\dext{\C}{(e^{\pm\ell_r})_{r\ge1}}.
\end{eqnarray}
Since $\partial\ell_r^{-1}=-\ell_r^{-2}\partial\ell_r$ then $\calL=\C[\{\ell_r^{\pm1},\partial^i\ell_r\}_{r,i\ge1}]$. Let
\begin{eqnarray}
\calL^+:=\C[\{\partial^i\ell_r\}_{r,i\ge1}].
\end{eqnarray}
This $\C$-differential subalgebra $\calL^+$ is an integral domain
generated by holomorphic functions and $\mathrm{Frac}(\calL^+)$ is generated by meromorphic functions. Since there is
$0\neq q_{i,l,k}\in\calL^+$ such that $(\partial^ie^{\pm\ell_k})^l=q_{i,l,k}e^{\pm l\ell_k}$ ($i,l,k\ge1$) then let
\begin{eqnarray}\label{E+}
\calE^+
&:=&\span_{\C}\{
(\partial^{i_1}e^{\pm\ell_{r_1}})^{l_1}\ldots(\partial^{i_k}e^{\pm\ell_{r_k}})^{l_k}\}_{
(i_1,l_1,r_1),\ldots,(i_k,l_k,r_k)\in(\N^*)^3,k\ge1}\cr
&=&\span_{\C}\{q_{i_1,l_1,r_1}\ldots q_{i_k,l_k,r_k}e^{l_1\ell_{r_1}+\ldots+l_k\ell_{r_k}}\}_{
(i_1,l_1,r_1),\ldots,(i_k,l_k,r_k)\in\N^*\times\Z^*\times\N^*,k\ge1}\cr
&\subset&\span_{\calL^+}\{
e^{l_1\ell_{r_1}+\ldots+l_k\ell_{r_k}}\}_{(l_1,r_1),\ldots,(l_k,r_k)\in\Z^*\times\N^*,k\ge1}\cr
&=:&\calC.\label{vide}
\end{eqnarray}
Note that in \eqref{vide}, $\calC$ is a differential subring of $\calA=\calH(\Omega)$
(hence, $\mathrm{Frac}(\calC)$ is a differential subfield of $\mathrm{Frac}(\calA)$) and
\begin{eqnarray}\label{cap2}
\calE^+\cap E=\{0\}.
\end{eqnarray}

\begin{theorem}\label{independence}
\begin{enumerate}
\item The family $(e^{\ell_r})_{r\ge1}$ (resp. $(\ell_r)_{r\ge1}$) is algebraically free over $\calE^+$ (resp. $\calL^+$).
\item $\C[E]$ and $\C[L]$ are algebraically disjoint, within $\calA$.
\end{enumerate}
\end{theorem}

\begin{proof}
\begin{enumerate}
\item Considering the Chen series of the differential forms $\{\omega_r\}_{r\ge1}$
defined, for any $r\ge1$, by $u_{y_r}=e^{\ell_r}\partial\ell_r$.
Let $Q\in\mathrm{Frac}(\calL)\cap E$ (resp. $\mathrm{Frac}(\calC)\cap E$):
\begin{enumerate}
\item since $Q\in E$ then there is $\{c_y\}_{y\in Y}\in\C^{(Y)}$ such that
\begin{eqnarray}\label{test}
Q=\sum_{r\ge1}c_{y_r}e^{\ell_r}&\mbox{and then}&\partial Q=\sum_{r\ge1}c_{y_r}e^{\ell_r}\partial\ell_r,
\end{eqnarray}

\item since $Q\in\mathrm{Frac}(\calL)\supset\calL\supset\C[L]$ (resp. $\mathrm{Frac}(\calC)\supset\calC\supset\calE^+$)
then, by \eqref{cap1} (resp. \eqref{cap2}) it remains that $Q=0$.
\end{enumerate}
Hence, by Proposition \ref{plein}, since $\{e^{\ell_k}\}_{k\ge1}$ is $\C$-free and $Q=0$ then
\begin{enumerate}
\item on the one hand, for any $r\ge1$, one has $c_{y_r}=0$,
\item on the other hand, $\{\alpha_0^z(S_l)\}_{l\in\Lyn Y}$ (including $\{\alpha_0^z(S_y)\}_{y\in Y}$)
is algebraically free over $\calL$ (resp. $\calC$).
\end{enumerate}
It follows that $\{e^{\ell_r}\}_{r\ge1}$ is algebraically free over $\C[L]$ (resp. $\calE^+$).

Now, suppose there is an algebraic relation among $(\ell_k)_{k\ge1}$ over $\calL^+$ in which,
by differentiating and substituting $\partial\ell_k$ by $e^{-\ell_k}\partial e^{\ell_k}$, we get
an algebraic relation among $\{e^{\ell_r}\}_{r\ge1}$ over $\C[L]$ and $\calE^+$ contradicting
with previous results. It follows then $(\ell_k)_{k\ge1}$ is $\calL^+$-algebraically independent.

\item $\{e^{\ell_k}\}_{k\ge1}$ (resp. $\{\ell_k\}_{k\ge1}$) is algebraically independent over $\C[L]$ (resp. $\C[E]$).
Hence, $\{e^{\ell_k},\ell_k\}_{k\ge1}$ generates freely $\C[E+L]$ and $\C[E]\cap\C[L]=\C.1_{\calH(\Omega)}$.

It follows that $\C[E]$ and $\C[L]$ are algebraically disjoint, within $\calA$.
\end{enumerate}
\end{proof}

\begin{corollary}\label{disjoint}
\begin{enumerate}
\item Using the inputs $\{\partial\ell_r\}_{r\ge1}$ (resp. $\{e^{\ell_r}\partial\ell_r\}_{r\ge1}$),
the following morphism is injective (see also Remark \eqref{Gamma})
\begin{eqnarray*}
\alpha_0^z:(\ncp{\calL^+}{Y},\shuffle,1_{Y^*})&\longrightarrow&
(\mathrm{span}_{\calL^+}\{\alpha_0^z(w)\}_{w\in Y^*},\times,1_{\calH(\Omega)}),\\
(\mbox{resp. }\alpha_0^z:(\ncp{\calE^+}{Y},\shuffle,1_{Y^*})
&\longrightarrow&(\mathrm{span}_{\calE^+}\{\alpha_0^z(w)\}_{w\in Y^*},\times,1_{\calH(\Omega)})).
\end{eqnarray*}

Hence, $\{\alpha_0^z(w)\}_{w\in Y^*}$ (resp. $\{\alpha_0^z(l)\}_{l\in\Lyn Y}$) is linearly (resp. algebraically)
independent over $\calL^+$ (resp. $\calE^+$).
\item Using the inputs $\{\partial\ell_r\}_{r\ge1}$ and denoting the set of exchangeable polynomials
(over $Y$ and with coefficients in $\C$) by $\ncp{\C_\mathrm{exc}}{Y}$ (see \cite{PVNC} for example),
the family $\{\alpha_0^z(\lambda)\}_{\lambda\in\Lyn Y\cup\{y_r^*\}_{r\ge1}}$ is $\C$-algebraically
independent and the restricted morphism
$\alpha_0^z:(\ncp{\C_\mathrm{exc}}{Y}\shuffle\C[\{y_r^*\}_{r\ge1}],\shuffle,1_{Y^*})\longrightarrow\C[L+E]$
is bijective.

Hence, $\{(e^{\ell_r})_{r\ge1},(\ell_r)_{r\ge1}\}$ is $\C$-algebraically independent.
\item Let $\calC_k:=\span_{\calL^+}
\{e^{l_1\ell_{r_1}+\ldots+l_k\ell_{r_k}}\}_{(l_1,r_1),\ldots,(l_k,r_k)\in\Z^*\times\N^*}$. Then
\begin{eqnarray*}
\calC=\bigoplus_{k\ge1}\calC_k
\end{eqnarray*}
\end{enumerate}
\end{corollary}

\begin{proof}
\begin{enumerate}
\item It is a consequence of Theorem \ref{independence}.
\item The free algebras $(\ncp{\C_{\mathrm{exc}}}{Y},\shuffle,1_{Y^*})$ and $(\C[\{y_r\}_{r\ge1}],\shuffle,1_{Y^*})$
are algebraically disjoint and their images by $\alpha_0^z$, by Proposition \ref{plein}, are, respectively, the 
free algebras $\C[L]$ and $\C[E]$ which are, by Theorem \ref{independence}, algebraically disjoint. Moreover,
since $\ncp{\C_{\mathrm{exc}}}{Y}=\C[\{y\}_{y\in Y}]$ and $Y\subset\Lyn Y$ then we deduce the respected results.
\item For any $k\ge1$, let $\Phi_k:=\span_{\C}\{e^{l_1\ell_{r_1}+\ldots+l_k\ell_{r_k}}\}_{
\mbox{\small distinct }r_1,\ldots,r_k\in\N^*,l_1,\ldots,l_k\in\Z^*}$.
Let $\C[\Phi]$ be the algebra of $\Phi:=\span_{\C}\{e^{\pm\ell_r}\}_{r\ge1}$. 
Since $(\ell_r)_{r\ge1}$ is $\C$-free then $\Phi_1\subsetneq\Phi_2\subsetneq\ldots$ and then
$\C[\Phi]=\bigoplus_{k\ge1}\Phi_k$. Moreover, the disjunction of $\C[E]$ and $\C[L]$ leads to 
$\calC_k\cong\calL^+\otimes_{\C}\Phi_k$ and then yields the expected result.
\end{enumerate}
\end{proof}

\begin{remark}
Let us back the second point of Proposition \ref{Weierstrass3} and then the formula \eqref{OrbFact},
for any $q\in\N_{\ge1}$ such that $q\equiv1(\mod 2)$, 
\begin{eqnarray*}
e^{\ell_{qr}(z)}=e^{\ell_r(z)}\prod_{\chi\in\mathbb{X}\setminus\{1\}}e^{\ell_r(\chi z)}.
\end{eqnarray*}
Since $(e^{\ell_r})_{r\ge1}$ is algebraically free over $\calE^+$ then
\begin{eqnarray*}
\prod_{\chi\in\mathbb{X}\setminus\{1\}}e^{\ell_r(\chi z)}\notin\calE^+[(e^{\ell_k})_{k\ge1}].
\end{eqnarray*}
\end{remark}

\subsection{Polylogarithms and harmonic sums indexed by rational series}\label{Polylogarithms}
Using the projector $\pi_X:(\CY,.,1_{Y^*})\longrightarrow(\CX,.,1_{X^*})$ defined as the concatenation morphism,
mapping $y_s$ to $x_0^{s-1}x_1$ and admitting $\pi_Y$ as adjoint, one has the following one-to-one correspondences
\begin{eqnarray*}
({s_1},\ldots,{s_r})\in(\N^*)^r\leftrightarrow y_{{s_1}}\ldots y_{{s_r}}\in Y^*
\mathop{\rightleftharpoons}_{\pi_Y}^{\pi_X}x_0^{{s_1}-1}x_1\ldots x_0^{{s_r}-1}x_1\in X^*x_1.
\end{eqnarray*}
In all the sequel, $\Omega:=\widetilde{\C\setminus\{0,1\}}$ and
\begin{eqnarray}\label{polylog}
\omega_0(z):=z^{-1}dz&\mbox{and}&.\omega_1(z):=(1-z)^{-1}dz.
\end{eqnarray}
By \eqref{notation}, $\Li_{{s_1},\ldots,{s_r}}(z)=\alpha_0^z(x_0^{{s_1}-1}x_1\ldots x_0^{{s_k}-1}x_1)$, for
${s_1},\ldots,{s_r}\in\N^*$. Thus, putting $\Li_{{x_0}}(z):=\log(z)$, the following morphisms are injective
\begin{eqnarray}
\Li_{\bullet}:(\ncp{\Q}{X},\shuffle,1_{X^*})&\longrightarrow&\left(\Q\{\Li_w\}_{w\in X^*},.,1\right),\cr
x_0^{{s_1}-1}x_1\ldots x_0^{{s_r}-1}x_1&\longmapsto&\Li_{x_0^{{s_1}-1}x_1\ldots x_0^{{s_r}-1}x_1}=\Li_{{s_1},\ldots,{s_r}},\label{Li}\\
\H_{\bullet}:(\ncp{\Q}{Y},\stuffle,1_{Y^*})&\longrightarrow&\left(\Q\{\H_w\}_{w\in Y^*},.,1\right),\cr
y_{{s_1}}\ldots y_{{s_r}}&\longmapsto&\H_{y_{{s_1}}\ldots y_{{s_r}}}=\H_{{s_1},\ldots,{s_r}}. \label{H}
\end{eqnarray}

In order to extend $\Li_{\bullet},\H_{\bullet}$ (in \eqref{Li}, \eqref{H})
over some subdomain of $\ncs{\C^{\mathrm{rat}}}{X}$ (resp. $\ncs{\C^{\mathrm{rat}}}{Y}$), let us call
$\mathrm{Dom}_R(\Li_{\bullet})$ the set of series 
\begin{eqnarray}\label{graduer}
S=\sum_{n\geq 0}S_n&\mbox{with}&S_n:=\sum_{|w|=n}\scal{S}{w}w
\end{eqnarray}
such that $\sum_{n\ge0}\Li_{S_n}$ converge uniformly in any compact of $\Omega$.

For any $0< R\le 1$, such that $\sum_{n\geq 0}\Li_{S_n}$ converge uniformly in the open disc $D_{\abs{z}<R}$,
one has $(1-z)^{-1}\Li_S=\sum_{N\ge0}a_Nz^N$ converge in the same disc and then $\H_{\pi_Y(S_n)}(N)=a_N$, for $N\ge0$.
Hence, let us define
\begin{eqnarray*}
\mathrm{Dom}_R(\Li_{\bullet})&:=&\{S\in\C1_{X_*}\oplus\ncs{\C}{X}x_1|\sum_{n\ge1}\Li_{S_n}\mbox{ converge in }D_{\abs{z}<R}\},\\
\mathrm{Dom}(\H_{\bullet})&:=&\pi_Y\mathrm{Dom}^{\mathrm{loc}}(\Li_{\bullet}),\mbox{ where }
\mathrm{Dom}^{\mathrm{loc}}(\Li_{\bullet}):=\bigcup_{0< R\le 1}\mathrm{Dom}_R(\Li_{\bullet}).
\end{eqnarray*}
Under suitable convergence condition this extension can be realized and \cite{Domains,Ngo2,CM}
\begin{enumerate}
\item $\mathrm{Dom}(\Li_{\bullet})$ (resp. $\mathrm{Dom}(\H_{\bullet})$) is closed by shuffle (resp. quasi-shuffle) products.
\item $\Li_{S\shuffle T}=\Li_S\Li_T$ and $\H_{S\stuffle T}=\H_S\H_T$,
for $S,T\in\mathrm{Dom}(\Li_{\bullet})$ (resp. $\mathrm{Dom}(\H_{\bullet})$).
\end{enumerate}
Any series $S\in\ncs{\C}{{\calX}}$ is syntactically exchangeable iff it is of the form
\begin{eqnarray}
S=\sum_{\alpha\in\N^{({\calX})},\mathrm{supp}(\alpha)=\{x_1,\ldots,x_k\}}
s_{\alpha}x_1^{\alpha(x_1)}\shuffle\ldots\shuffle x_k^{\alpha(x_k)}.
\end{eqnarray}
The set of these series, a $\shuffle$-subalgebra of $\ncs{A}{\calX}$,
will be denoted by $\ncs{\C_{\mathrm{exc}}^{\mathrm{synt}}}{\calX}$.

\begin{theorem}[extension of $\Li_{\bullet}$]\label{Indexation}
Let $\calC_{\C}:=\C[\{z^a,(1-z)^b\}_{a,b\in\C}]$. Then
\begin{enumerate}
\item The algebra ${\calC_{\C}}\{\Li_w\}_{w\in X^*}$ is closed under the differential operators
$\theta_0:=z{\partial_z}$ and $\theta_1:=(1-z){\partial_z}$ and under their sections ${\iota_0},{\iota_1}$
($\theta_0\iota_0=\theta_1\iota_1=\mathrm{Id}$).

\item The bi-integro differential algebra
$(\calC_{\C}\{\Li_w\}_{w\in X^*},\theta_0,\theta_1,\iota_0,\iota_1)$
is closed under the action of the group of transformations, ${\calG}$,
generated by $\{z\mapsto1-z,z\mapsto1/z\}$, permuting $\{0,1,+\infty\}$:
\begin{eqnarray*}
\forall h\in{\calC_{\C}}\{\Li_w\}_{w\in X^*},&\forall g\in{\calG},&h(g)\in{\calC_{\C}}\{\Li_w\}_{w\in X^*}.
\end{eqnarray*}

\item If $R\in{\ncs{{\mathbb C}_{\mathrm{exc}}^{\mathrm{rat}}}{X}}\shuffle\CX$
(resp. ${\ncs{{\mathbb C}_{\mathrm{exc}}^{\mathrm{rat}}}{X}}$)
then $\Li_R\in{\calC_{\C}}\{\Li_w\}_{w\in X^*}$ (resp. ${\calC_{\C}}[\log(z),\log(1-z)]$).

\item The family $\{\Li_w\}_{w\in X^*}$ (resp. $\{\Li_l\}_{l\in\Lyn X}$)
is linearly (resp. algebraically) independent over ${\calC_{\C}}$.
\end{enumerate}
\end{theorem}

\goodbreak

\begin{proof}
The three first items are immediate. Only the last one needs a proof:

Let then $B=\C\setminus\{0,1\}$, $\Omega=\C\setminus(]-\infty,0]\cup[1,+\infty[)$
and choose a basepoint  $b\in\Omega$, one has the following diagram
\begin{eqnarray*}
\begin{tikzcd}
&(\tilde{B},\tilde{b})\ar{d}{p}\\
(\Omega,b)\ar[hook]{r}{j}\ar[hook]{ur}{s}&(B,b)
\end{tikzcd}
\end{eqnarray*}
Any holomorphic function $f\in \calH(\Omega)$ such that $f'=df/dz$ admits an analytic continuation to $B$ 
can be lifted to $\tilde{B}$ by $\tilde{f}(z):=f(b)+\int_{b}^zf'(s)ds$.

Let $\L$ be the noncommutative series of the polylogarithms $\{\Li_w\}_{w\in X^*}$, which is group-like,
and $C_{z_0\path z}$ be the Chen series, of $\{\omega_0,\omega_1\}$ along $z_0\path z\in\tilde{B}$,
$C_{z_0\path z}=\L(z)\L^{-1}(z_0)$ (see \cite{CM}).

Now, in view of Lemma \ref{ind_lin}, as the algebra $\calC$ is without zero divisors and contains the field of constants $\C$,
it suffices to prove that $\Li_{x_0},\Li_{x_1}$ and $1_\Omega$ are $\calC$-linearly independent. It is an easy exercise to check
that $s_*(\tilde{f}):=\tilde{f}\circ s$ coincides with the given $f$. This is the case, in particular of the functions
$\log(z)$ and $\log((1-z)^{-1})$ whose liftings will be denoted $\log_0$ and $\log_1$, respectively.
So, we lift the functions $z^a$ and $(1-z)^b$ as, respectively,
\begin{eqnarray}
e_0^a(\tilde{z}):=e^{a\log_0(\tilde{z})}&\mbox{and}&e_1^b:=e^{b \log_1(\tilde{z})}
\end{eqnarray}
and, of course, by construction,
\begin{eqnarray}\label{lifting1}
e_0^a\circ s=(z\mapsto z^a)&\mbox{and}&e_1^b\circ s=(z\mapsto(1-z)^b) 
\end{eqnarray}
We suppose a dependence relation, in $\calH(\Omega)$ 
\begin{eqnarray}\label{rel1}
P_0(z^a,(1-z)^b)\Li_{x_0}+P_1(z^a,(1-z)^b)\Li_{x_1}+P_2(z^a,(1-z)^b).1_\Omega=0
\end{eqnarray}
where $P_i\in \C[X,Y]$ are two-variable polynomials.
From \eqref{lifting1} and the fact that $\Omega\not=\emptyset$, we get   
\begin{eqnarray}
P_0(e_0^a,e_1^b)\log_0+P_1(e_0^a,e_1^b)\log_1+P_2(e_0^a,e_1^b).1_{\tilde{B}}=0.
\end{eqnarray}

Now, we consider $D_0$ (resp. $D_1$), the deck transformation corresponding to the path 
$\sigma_0(t)=e^{2\mathrm{i}\pi t}/2$ (resp. $\sigma_1(t)=(1-e^{-2\mathrm{i}\pi t})/2$, one gets 
\begin{eqnarray}\label{mono1}
\log_0\circ(D_0^r)(\tilde{z})=\log_0(\tilde{z})+2\mathrm{i}r\pi
&\mbox{and}&
\log_1\circ(D_1^s)(\tilde{z})=\log_1(\tilde{z})+2\mathrm{i}s\pi
\end{eqnarray}
Now we remark that 
\begin{eqnarray}
e_0^{[a]}\circ D_0(\tilde{z})=e_0^{[a]}(\tilde{z})e^{2a\mathrm{i}\pi}&\mbox{and}&e_1^{[b]}\circ D_0=e_1^{[b]}
\end{eqnarray}
and, similarly
\begin{eqnarray}
e_1^{[b]}\circ D_1(\tilde{z})=e_1^{[b]}(\tilde{z})e^{2b\mathrm{i}\pi}&\mbox{and}&e_0^{[a]}\circ D_1=e_0^{[a]}
\end{eqnarray}
so that $P_i(e_0^a,e_1^b)$ remain bounded through the actions of $D_0^r$ and $D_1^s$,
from \eqref{mono1}, we get that $P_i=0,i=0..2$ which proves the claim.
\end{proof}

\begin{example}[\cite{Ngo2}]
Let us use the noncommutative multivariate exponential transforms
{\it i.e.}, for any syntactically exchangeable series, we get
\begin{eqnarray*}
\sum_{i_0,i_1\ge0}s_{i_0,i_1}x_0^{i_0}\shuffle x_1^{i_1}&\longmapsto&
\sum_{i_0,i_1\ge0}\dfrac{s_{i_0,i_1}}{i_0!i_1!}\Li_{x_0}^{i_0}\Li_{x_1}^{i_1}.
\end{eqnarray*}
Hence, $x_0^n\longmapsto{\Li_{x_0}^n}/{n!}$ and $x_1^n\longmapsto{\Li_{x_1}^n}/{n!}$,
for $n\in\N$, yielding some polylogarithms indexed by series,
\begin{eqnarray*}
\Li_{x_0^*}(z)=z,&\Li_{x_1^*}(z)=(1-z)^{-1},&\Li_{(ax_0+bx_1)^*}(z)=z^a(1-z)^{-b}.
\end{eqnarray*}

Moreover, for any $(s_1,\ldots,s_r)\in\N^r_+$, there exists an unique series ${R_{y_{s_1}\ldots y_{s_r}}}$
belonging to $(\Z[x_1^*],\shuffle,1_{X^*})$ such that $\Li_{-{s_1},\ldots,-{s_r}}=\Li_{R_{y_{s_1}\ldots y_{s_r}}}$.
More precisely (by convention $\rho_0=x_1^*-1_{X^*}$),
\begin{eqnarray*}
R_{y_{s_1}\ldots y_{s_r}}=
\sum_{k_1=0}^{s_1}\ldots\sum_{k_r=0}^{(s_1+\ldots+s_r)-(k_1+\ldots+k_{r-1})}
\binom{s_1}{k_1}\ldots\binom{\sum_{i=1}^rs_i-\sum_{i=1}^{r-1}k_i}{k_r}\\
\rho_{k_1}\shuffle\ldots\shuffle\rho_{k_r},
\end{eqnarray*}
and using the Stirling numbers of second kind, $S_2({k_i},j)$, one has
\begin{eqnarray*}
\rho_{k_i}=x_1^*\shuffle\sum_{j=1}^{k_i}S_2({k_i},j)j!(x_1^*-1_{X^*})^{\shuffle j}, (k_i\neq0).
\end{eqnarray*}
\end{example}

\begin{theorem}[extension of $\H_{\bullet}$]\label{Indexation2}
For any $r\ge1$, one has, for any $t\in\C,\abs{t}<1$,
\begin{eqnarray*}
\H_{(t^ry_r)^*}=\sum_{k\ge0}\H_{y_r^k}t^{kr}
=\exp\biggl(\sum_{k\ge1}\H_{y_{kr}}\frac{(-t^r)^{k-1}}k\biggr).
\end{eqnarray*}
Moreover, for $\abs{a_s}<1,\abs{b_s}<1$ and $\abs{a_s+b_s}<1$,
\begin{eqnarray*}
\H_{(\sum_{s\ge1}(a_s+b_s)y_s+\sum_{r,s\ge1}a_sb_ry_{s+r})^*}
=\H_{(\sum_{s\ge1}a_sy_s)^*}\H_{(\sum_{s\ge1}b_sy_s)^*}.
\end{eqnarray*}
Hence,
\begin{eqnarray*}
\H_{(a_sy_s+a_ry_r+a_sa_ry_{s+r})^*}=\H_{(a_sy_s)^*}\H_{(a_ry_r)^*},&
\H_{(-a_s^2y_{2s})^*}=\H_{(a_sy_s)^*}\H_{(-a_sy_s)^*}.
\end{eqnarray*}
\end{theorem}

\begin{proof}
For $t\in\C,\abs{t}<1$, since $\Li_{(tx_1)^*}$ is well defined then
so is the arithmetic function, expressed via Newton-Girard formula (see \cite{JSC}), for $n\ge0$, by
\begin{eqnarray*}
\H_{(ty_1)^*}(n)
=\sum_{k\ge0}\H_{y_1^k}(n)t^{k}
=\exp\Bigl(-\sum_{k\ge1}\H_{y_k}(n)\frac{(-t)^{k}}k\Bigr)
=\prod_{l=1}^n\Bigl(1+\frac{t}{l}\Bigr).
\end{eqnarray*}
Similarly, for any $r\ge2$, the transcendent function $\H_{(t^ry_r)^*}$ can be expressed
via Newton-Girard formula (see \cite{JSC}) once again and via Adam's transform, by
\begin{eqnarray*}
\H_{(t^ry_r)^*}(n)
=\sum_{k\ge0}\H_{y_r^k}(n)t^{kr}
=\exp\Bigl(-\sum_{k\ge1}\H_{y_{kr}}(n)\frac{(-t^r)^{k}}k\Bigr)
=\prod_{l=1}^N\Bigl(1-\frac{(-t^r)}{l^r}\Bigr).
\end{eqnarray*}
Since $\absv{\H_{y_r}}_{\infty}\le\zeta(r)$ then
$-\sum_{k\ge1}{\H_{kr}{(-t^r)^{k}}}/{k}$ is termwise dominated by $\absv{\ell_r}_{\infty}$
and then $\H_{(t^ry_r)^*}$ by $e^{\ell_r}$ (see also Theorem \ref{Newton-Girard} bellow).
It follows then the last results by using the following identity \cite{PVNC}
\begin{eqnarray*}
\Bigl(\sum_{s\ge1}a_sy_s\Bigr)^*\stuffle\Bigl(\sum_{s\ge1}b_sy_s\Bigr)^*=
\Bigl(\sum_{s\ge1}(a_s+b_s)y_s+\sum_{r,s\ge1}a_sb_ry_{s+r}\Bigr)^*.
\end{eqnarray*}
\end{proof}

From the estimations from above of the previous proof, it follows then
\begin{corollary}
For any $r\ge2$, one has
\begin{eqnarray*}
\frac{1}{\Gamma_{y_r}(1+t)}
=\sum_{k\ge0}\zeta(\underbrace{{r,\ldots,r}}_{k{\tt times}})t^{kr}
=\exp\Bigl(-\sum_{k\ge1}\zeta(kr)\frac{(-t^r)^{k}}k\Bigr)
=\prod_{n\ge1}\Bigl(1-\frac{(-t^r)}{n^r}\Bigr).
\end{eqnarray*}
\end{corollary}

\begin{corollary}
For any $r\ge1$
\begin{eqnarray*}
y_r^*=\exp_{\stuffle}\Bigl(\sum_{k\ge1}y_{kr}\frac{(-1)^{k-1}}k\Bigr).
\end{eqnarray*}
Hence, for any $k\ge0$, one has
\begin{eqnarray*}
y_r^n&=&\frac{(-1)^n}{n!}\sum_{s_1,\ldots,s_n>0\atop s_1+\ldots+ns_n=n}
\frac{(-y_r)^{\stuffle s_1}}{1^{s_1}}\stuffle\ldots\stuffle\frac{(-y_{nr})^{\stuffle s_n}}{n^{s_n}}
\end{eqnarray*}
and, for any $r,s\ge1$, one also has
\begin{eqnarray*}
y_r^*\stuffle y_r^*&=&\sum_{k=0}^r{r+s-k\choose s}{s\choose k}y_{r+s-k}.
\end{eqnarray*}
\end{corollary}

\subsection{Extended double regularization by Newton-Girard formula}\label{regularization}

By \eqref{Li}--\eqref{H}, the following polymorphism is, by definition, surjective (see \cite{CM}) 
\begin{eqnarray}
\begin{matrix}
\zeta:{\displaystyle(\Q 1_{X^*}\oplus x_0\QX x_1,\shuffle,1_{X^*})\atop
\displaystyle(\Q 1_{Y^*}\oplus(Y-\{y_1\})\ncp{\Q}{Y},\stuffle,1_{Y^*})}
&\longtwoheadrightarrow&({\calZ},.,1),
\end{matrix}
\end{eqnarray}
mapping both $x_0^{{s_1}-1}x_1\ldots x_0^{{s_r}-1}x_1$ and $y_{s_1}\ldots y_{s_r}$ to
$\zeta({s_1},\ldots,{s_r})$, where $\calZ$ denotes the $\Q$-algebra (algebraically) generated
by $\{\zeta(l)\}_{l\in\Lyn X-X}$, or equivalently, $\{\zeta(l)\}_{l\in\Lyn Y-\{y_1\}}$.
It can be extended as characters
\begin{eqnarray}
\zeta_{\shuffle}:(\ncp{\R}X,\shuffle,1_{X^*})&\longrightarrow&(\R,.,1),\\
\zeta_{\stuffle},\gamma_{\bullet}:(\ncp{\R}Y,\stuffle,1_{Y^*})&\longrightarrow&(\R,.,1)
\end{eqnarray}
such that, for any $l\in\Lyn X$, one has (see \cite{CM})
\begin{eqnarray}
\zeta_{\shuffle}(l)=&\mathrm{f.p.}_{z\rightarrow1}\Li_l(z),&\{(1-z)^a\log^b(1-z)\}_{a\in\Z,b\in\N},\\
\zeta_{\stuffle}(\pi_Yl)=&\mathrm{f.p.}_{n\rightarrow+\infty}\H_{\pi_Yl}(n),&\{n^a\H_1^b(n)\}_{a\in\Z,b\in\N},\\
\gamma_{\pi_Yl}=&\mathrm{f.p.}_{n\rightarrow+\infty}\H_{\pi_Yl}(n),&\{n^a\log^b(n)\}_{a\in\Z,b\in\N}.
\end{eqnarray}
It follows that, for any $l\in\Lyn X-X$, $\zeta_{\shuffle}(l)=\zeta_{\stuffle}(\pi_Yl)=\gamma_{\pi_Yl}={\zeta(l)}$,
and, for the algebraic generator $x_0$, $\zeta_{\shuffle}(x_0)=0=\log(1)$
and, for the algebraic generators $x_1$ and $y_1$ (divergent cases),
\begin{eqnarray}
\zeta_{\shuffle}(x_1)=0=\mathrm{f.p.}_{z\rightarrow1}\log(1-z),&&\{(1-z)^a\log^b(1-z)\}_{a\in\Z,b\in\N},\\
\zeta_{\stuffle}(y_1)=0=\mathrm{f.p.}_{n\rightarrow+\infty}\H_1(n),&&\{n^a\H_1^b(n)\}_{a\in\Z,b\in\N},\\
\gamma_{y_1}=\gamma=\mathrm{f.p.}_{n\rightarrow+\infty}\H_1(n),&&\{n^a\log^b(n)\}_{a\in\Z,b\in\N}.
\end{eqnarray}

As in \cite{Ngo2,CM}, considering a character $\chi_{\bullet}$ on $(\ncp{\C}{X},\shuffle,1_{X^*})$
and considering $\mathrm{Dom}(\chi_{\bullet})\subset\ncs{\C}{X}$ as in \eqref{graduer},
we can also check easily that \cite{Domains}:
\begin{itemize}
\item $\ncp{\C}{X}\shuffle\ncs{\C^{\mathrm{rat}}}{X}\subset\mathrm{Dom}(\chi_{\bullet})$ which is closed by shuffle product,
\item for any $S,T\in\mathrm{Dom}(\chi_{\bullet})$, one has $\chi_{S\shuffle T}=\chi_S\chi_T$,
\item if $S\in\mathrm{Dom}(\chi_{\bullet})$ then $\exp_{\shuffle}(S)\in\mathrm{Dom}(\chi_{\bullet})$
and $\chi_{\exp_{\shuffle}(S)}=e^{\chi_S}$.
\end{itemize}

Similarly, considering a character $\chi_{\bullet}$ on $(\ncp{\C}{Y},\stuffle,1_{Y^*})$
and considering $\mathrm{Dom}(\chi_{\bullet})\subset\ncs{\C}{Y}$ as in \eqref{graduer},
we can also check easily that \cite{Domains}:
\begin{itemize}
\item $\ncp{\C}{Y}\stuffle\ncs{\C^{\mathrm{rat}}}{Y}\subset\mathrm{Dom}(\chi_{\bullet})$ which is closed by quasi-shuffle product,
\item for any $S,T\in\mathrm{Dom}(\chi_{\bullet})$, one has $\chi_{S\stuffle T}=\chi_S\chi_T$,
\item if $S\in\mathrm{Dom}(\chi_{\bullet})$ then $\exp_{\stuffle}(S)\in\mathrm{Dom}(\chi_{\bullet})$
and $\chi_{\exp_{\stuffle}(S)}=e^{\chi_S}$.
\end{itemize}

\begin{example}
For any $z\in\C,\abs{z}<1,x\in X=\{x_0,x_1\},y_r\in Y=\{y_k\}_{k\ge1}$, since
\begin{eqnarray*}
(zx)^*=\exp_{\shuffle}(z)&\mbox{and}&(zy_r)^*=\exp_{\stuffle}\Bigl(\sum_{k\ge1}y_{kr}\frac{(-z)^{k-1}}k\Bigr)
\end{eqnarray*}
then
\begin{eqnarray*}
\zeta_{\shuffle}((zx)^*)=e^{z\zeta_{\shuffle}(x)}&\mbox{and}&
\gamma_{(zy_r)^*}=\exp\Bigl(\sum_{k\ge1}\zeta_{\stuffle}(y_{kr})\frac{(-z)^{k-1}}k\Bigr).
\end{eqnarray*}
\end{example}

We now in situation to state that

\begin{theorem}[Regularization by Newton-Girard formula]\label{Newton-Girard}
The characters $\zeta_{\shuffle}$ and $\gamma_{\bullet}$ are extended algebraically as follows:
\begin{eqnarray*}
\begin{matrix}
\zeta_{\shuffle}:(\CX\shuffle\ncs{\C_{\mathrm{exc}}^{\mathrm{rat}}}{X},\shuffle,1_{X^*})&\longrightarrow&(\C,.,1),\\
\forall t\in\C,\abs{t}<1,(tx_0)^*,(tx_1)^*&\longmapsto &{1_{\C}}.
\end{matrix}
\end{eqnarray*}
and 
\begin{eqnarray*}
 \begin{matrix}
\gamma_{\bullet}:(\CY\stuffle\ncs{\C_{\mathrm{exc}}^{\mathrm{rat}}}{Y},\stuffle,1_{Y^*})&\longrightarrow&(\C,.,1),\\
\forall t\in\C,\abs{t}<1,\forall r\ge1,(t^ry_r)^*&\longmapsto&\Gamma_{y_r}^{-1}(1+t).
\end{matrix}
\end{eqnarray*}
Moreover, the morphism $(\C[\{y_r^*\}_{r\ge1}],\shuffle,1_{Y^*})\longrightarrow\C[E]$
mapping $y_r^*$ to $\Gamma_{y_r}^{-1}$, is injective and
$\Gamma_{y_{2r}}(1+\sqrt[2r]{-1}t)=\Gamma_{y_r}(1+t)\Gamma_{y_r}(1+\sqrt[r]{-1}t)$, for $r\ge1$.
\end{theorem}

\begin{proof}
By Definition \ref{f+expf}, Propositions \ref{plein}, \ref{Weierstrass3} and Theorems
\ref{Indexation}, \ref{Indexation2}, we get the expected results (see also Proposition \ref{Weierstrass2}).
\end{proof}

\goodbreak

\begin{example}\cite{Ngo2}
\begin{eqnarray*}
\Li_{-1,-1}&=&\Li_{-x_1^*+5(2x_1)^*-7(3x_1)^*+3(4x_1)^*},\cr
\Li_{-2,-1}&=&\Li_{x_1^*-11(2x_1)^*+31(3x_1)^*-33(4x_1)^*+12(5x_1)^*},\cr
\Li_{-1,-2}&=&\Li_{x_1^*-9(2x_1)^*+23(3x_1)^*-23(4x_1)^*+8(5x_1)^*},\cr
\H_{-1,-1}&=&\H_{-y_1^*+5(2y_1)^*-7(3y_1)^*+3(4y_1)^*},\cr
\H_{-2,-1}&=&\H_{y_1^*-11(2y_1)^*+31(3y_1)^*-33(4y_1)^*+12(5y_1)^*},\cr
\H_{-1,-2}&=&\H_{y_1^*-9(2y_1)^*+23(3y_1)^*-23(4y_1)^*+8(5y_1)^*}.
\end{eqnarray*}
Hence,
$\begin{array}{ccc}
\zeta_{\shuffle}(-1,-1)=0,&
\zeta_{\shuffle}(-2,-1)=-1,&
\zeta_{\shuffle}(-1,-2)=0,
\end{array}$
and
\begin{eqnarray*}
\gamma_{-1,-1}&=&-\Gamma^{-1}(2)+5\Gamma^{-1}(3)-7\Gamma^{-1}(4)+3\Gamma^{-1}(5)={11}/{24},\cr
\gamma_{-2,-1}&=&\Gamma^{-1}(2)-11\Gamma^{-1}(3)+31 \Gamma^{-1}(4)-33\Gamma^{-1}(5)+12\Gamma^{-1}(6)=-{73}/{120},\cr
\gamma_{-1,-2}&=&\Gamma^{-1}(2)-9\Gamma^{-1}(3)+23 \Gamma^{-1}(4)-23\Gamma^{-1}(5)+8\Gamma^{-1}(6)=-{67}/{120}.
\end{eqnarray*}
\end{example}

From Theorems \ref{Indexation2} and \ref{Newton-Girard}, one deduces

\begin{corollary}\label{identity}
\begin{enumerate}
\item With the notations of \eqref{notation}, Definition \ref{f+expf} and with the
differential forms $\{(\partial\ell_r)dz\}_{r\ge1}$, for any $z\in\C,\abs{z}<1$, one has
\begin{eqnarray*}
\gamma_{\STUFFLE_{r\ge1}(z^ry_r)^*}
=\prod_{r\ge1}\gamma_{(z^ry_r)^*}=\prod_{r\ge1}e^{\ell_r(z)}=\prod_{r\ge1}\dfrac1{\Gamma_{y_r}(1+z)}=\alpha_0^z(\shuffle_{r\ge1}y_r^*).
\end{eqnarray*}
\item One has, for $\abs{a_s}<1,\abs{b_s}<1$ and $\abs{a_s+b_s}<1$,
\begin{eqnarray*}
\gamma_{(\sum_{s\ge1}(a_s+b_s)y_s+\sum_{r,s\ge1}a_sb_ry_{s+r})^*}
=\gamma_{(\sum_{s\ge1}a_sy_s)^*}\gamma_{(\sum_{s\ge1}b_sy_s)^*}.
\end{eqnarray*}
Hence,
$\gamma_{(a_sy_s+a_ry_r+a_sa_ry_{s+r})^*}=\gamma_{(a_sy_s)^*}\gamma_{(a_ry_r)^*},
\gamma_{(-a_s^2y_{2s})^*}=\gamma_{(a_sy_s)^*}\gamma_{(-a_sy_s)^*}$.
\end{enumerate}
\end{corollary}

\begin{remark}
The restriction $\alpha_0^z:(\C[\{y_r,y_r^*\}_{r\ge1}],\shuffle,1_{Y^*})\longrightarrow\C[L+E]$
is injective (see Corollary \ref{disjoint}) while $\ker(\gamma_{\bullet})\neq\{0\}$,
over $\ncp{\C}{Y}\stuffle\ncs{\C^{\mathrm{rat}}_{\mathrm{exc}}}{Y}$ \cite{CM}.
\end{remark}

\begin{example}\label{example1}\cite{words03,orlando}
By Theorem \ref{Newton-Girard},
\begin{eqnarray*}
\gamma_{(-t^2y_2)^*}=\Gamma_{y_2}^{-1}(1+\mathrm{i}t),&
\gamma_{(ty_1)^*}=\Gamma_{y_1}^{-1}(1+t),&
\gamma_{(-ty_1)^*}=\Gamma_{y_1}^{-1}(1+t).
\end{eqnarray*}
Then, by Corollary \ref{identity},
$\gamma_{(-t^2y_2)^*}=\gamma_{(ty_1)^*}\gamma_{(-ty_1)^*}$
meaning that
\begin{eqnarray*}
\Gamma_{y_2}^{-1}(1+\mathrm{i}t)=\Gamma_{y_1}^{-1}(1+t)\Gamma_{y_1}^{-1}(1-t).
\end{eqnarray*}
Or equivalently,
\begin{eqnarray*}
\exp\Bigl(-\sum_{k\ge2}\zeta(2k)\frac{t^{2k}}{k}\Bigr)
=\sum_{k\ge2}\zeta(\overbrace{{2,\ldots,2}}^{k{\tt times}})(-1)^kt^{2k}
=\frac{\sin(t\pi)}{t\pi}
=\sum_{k\ge1}(-1)^k\frac{(t\pi)^{2k}}{(2k+1)!}.
\end{eqnarray*}
Since
$\gamma_{(-t^2y_2)^*}=\zeta((-t^2x_0x_1)^*)$
then, identifying the coefficients of $t^{2k}$, we get
\begin{eqnarray*}
\frac{\zeta(\overbrace{{2,\ldots,2}}^{k{\tt times}})}{\pi^{2k}}=\dfrac{1}{(2k+1)!}\in\Q.
\end{eqnarray*}

Similarly, by Theorem \ref{Newton-Girard},
\begin{eqnarray*}
\gamma_{(-t^4y_4)^*}=\Gamma_{y_4}^{-1}(1+\sqrt[4]{-1}t),&
\gamma_{(t^2y_2)^*}=\Gamma_{y_2}^{-1}(1+t),&
\gamma_{(-t^ty_2)^*}=\Gamma_{y_2}^{-1}(1+\mathrm{i}t).
\end{eqnarray*}
Then, by Corollary \ref{identity},
$\gamma_{(-t^4y_4)^*}=\gamma_{(t^2y_2)^*}\gamma_{(-t^2y_2)^*}$
meaning that
\begin{eqnarray*}
\Gamma_{y_4}^{-1}(1+\sqrt[4]{-1}t)=\Gamma_{y_2}^{-1}(1+t)\Gamma_{y_2}^{-1}(1+\mathrm{i}t).
\end{eqnarray*}
Or equivalently,
\begin{eqnarray*}
\exp\Bigl(-\sum_{k\ge1}\zeta(4k)\frac{t^{4k}}{k}\Bigr)
&=&\sum_{k\ge2}\zeta(\underbrace{{4,\ldots,4}}_{k{\tt times}})(-1)^k\frac{(t\pi)^{4k}}{(4k+2)!}\\
&=&\frac{\sin(\mathrm{i}t\pi)}{\mathrm{i}t\pi}\frac{\sin(t\pi)}{t\pi}\\
&=&\sum_{k\ge1}\dfrac{2(-4t\pi)^{4k}}{(4k+2)!}.
\end{eqnarray*}
Since
$\begin{array}{ccc}
\gamma_{(-t^4y_4)^*}=\zeta((-t^4y_4)^*),&
\gamma_{(-t^2y_2)^*}=\zeta((-t^2y_2)^*),&
\gamma_{(t^2y_2)^*}=\zeta((t^2y_2)^*)
\end{array}$
then, using the poly-morphism $\zeta$ and identities on rational series, we get
\begin{eqnarray*}
\zeta((-t^4y_4)^*)&=&\zeta((-t^2y_2)^*)\zeta((t^2y_2)^*)\\
&=&\zeta((-t^2x_0x_1)^*)\zeta((t^2x_0x_1)^*))\\
&=&\zeta((-4t^4x_0^2x_1^2)^*).
\end{eqnarray*}
Thus, by identification the coefficients of $t^{4k}$, we obtain
\begin{eqnarray*}
\frac{\zeta(\overbrace{{4,\ldots,4}}^{k{\tt times}})}{4^k\pi^{4k}}
=\dfrac{\zeta(\overbrace{{3,1,\ldots,3,1}}^{k{\tt times}})}{\pi^{4k}}
=\frac{2}{(4k+2)!}\in\Q.
\end{eqnarray*}
\end{example}

\begin{corollary}[comparison formula, \cite{CASC2018,CM}]\label{C1}
For any $z,a,b\in\C$ such that $\abs{z}<1$ and $\Re(a)>0,\Re (b)>0$, we have
\begin{eqnarray*}
\mathrm{B}(z;a,b)
=\Li_{x_0[(ax_0)^*\shuffle((1-b)x_1)^*]}(z)
=\Li_{x_1[((a-1)x_0)^*\shuffle(-bx_1)^*]}(z).
\end{eqnarray*}
Hence, on the one hand
\begin{eqnarray*}
\mathrm{B}(a,b)
=\zeta_{\shuffle}(x_0[(ax_0)^*\shuffle((1-b)x_1
=\zeta_{\shuffle}(x_1[((a-1)x_0)^*\shuffle(-bx_1)^*])
\end{eqnarray*}
and, on the other hand
\begin{eqnarray*}
\mathrm{B}(a,b)
=\frac{\gamma_{((a+b-1)y_1)^*}}{\gamma_{((a-1)y_1)^*\stuffle((b-1)y_1)^*}}
=\frac{\gamma_{((a+b-1)y_1)^*}}{\gamma_{((a+b-2)y_1+(a-1)(b-1)y_2)^*}}.
\end{eqnarray*}
\end{corollary}

\begin{proof}
The results, of $\mathrm{B}(z;a,b)$, are the computations of iterated integrals associated
to different rational series, using the differential forms in \eqref{polylog}.
Those of $\mathrm{B}(a,b)$, are then immediate consequences,
by evaluating these iterated integrals at $z=1$ and by using Definition \ref{f+expf}
and Corollary \ref{identity}.
\end{proof}

\begin{example}\cite{words03,orlando}\label{example2}
Let us consider, for $t_0,t_1\in\C,\abs{t_0}<1,\abs{t_1}<1$,
\begin{eqnarray*}
R:=t_0x_0(t_0x_0+t_0t_1x_1)^*(t_0t_1x_1)=t_0^2t_1x_0[(t_0x_0)^*\shuffle(t_0t_1x_1)^*]x_1.
\end{eqnarray*}
Then, with the differential forms in \eqref{polylog}, we get successively
\begin{eqnarray*}
\Li_R(z)
&=&t_0^2t_1\int_0^z\frac{ds}s\int_0^s\Bigl(\frac{s}r\Bigr)^{t_0}\Bigl(\frac{1-r}{1-s}\Bigr)^{t_0t_1}\frac{dr}{1-r}\\
&=&t_0^2t_1\int_0^z(1-s)^{-t_0t_1}s^{t_0-1}\int_0^s(1-r)^{t_0t_1-1}r^{-t_0}dsdr
\end{eqnarray*}
By change of variable, $r=st$, we obtain then
\begin{eqnarray*}
\Li_R(z)&=&
t_0^2t_1\int_0^z(1-s)^{-t_0t_1}\int_0^1(1-st)^{t_0t_1-1}t^{-t_0}dtds.
\end{eqnarray*}
It follows then
\begin{eqnarray*}
\zeta(R)=t_0^2t_1\int_0^1\int_0^1(1-s)^{-t_0t_1}(1-st)^{t_0t_1-1}t^{-t_0}dtds.
\end{eqnarray*}
By change of variable, $y=\dfrac{1-s}{1-st}$, we obtain also
\begin{eqnarray*}
\zeta(R)=t_0^2t_1\int_0^1\int_0^1(1-ty)^{-1}t^{-t_0}y^{-t_0t_1}dydt.
\end{eqnarray*}
By expending $(1-ty)^{-1}$ and then integrating, we get on the one hand
\begin{eqnarray*}
\zeta(R)=\sum_{n\ge1}\frac{t_0}{n-t_0}\frac{t_0t_1}{n-t_0t_1}=\sum_{k>l>0}\zeta(k)t_0^kt_1^l.
\end{eqnarray*}
On the other hand, using the the expansion of $R$, we get also
\begin{eqnarray*}
\zeta(R)&=&\sum_{k>0}\sum_{l>0}\sum_{s_1+\ldots+s_l=k\atop s_1\ldots,s_l\ge1,s1\ge2}\zeta(s_1,\ldots,s_l)t_0^kt_1^l.
\end{eqnarray*}
Finally, by identification the coefficients of $\scal{\zeta(R)}{t_0^kt_1^l}$, we deduce the {\it sum formula}
\begin{eqnarray*}
\zeta(k)=\sum_{s_1+\ldots+s_l=k\atop s_1\ldots,s_l\ge1,s_1\ge2}\zeta(s_1,\ldots,s_l).
\end{eqnarray*}
\end{example}

\section{Conclusion}\label{conclusion}
In this work, we illustrated a bijection, between a sub shuffle algebra of noncommutative rational series
(recalled in \ref{combinatorialframeworks})
and a subalgebra of holomorphic functions, $\calH(\Omega)$, on a simply connected domain $\Omega\subset\C$ containing
the family of extended eulerian functions $\{\Gamma_y^{-1}(1+z)\}_{y\in Y}$ and the family of their logarithms,
$\{\log\Gamma_y^{-1}(1+z)\}_{y\in Y}$ (introduced in \ref{eulerianfunctions}), involved in summations of polylogarithms
and harmonics sums (studied in \ref{Polylogarithms}) and in regularizations of divergent polyzetas
(achieved, for this stage, in \ref{regularization}).

These two families are algebraically independent over a differential subring of $\calH(\Omega)$ and generate
freely two disjoint functional algebras. For any $y_r\in Y$, the special functions $\Gamma_{y_r}^{-1}(1+z)$
and $\log\Gamma_{y_r}^{-1}(1+z)$ are entire and holomorphic on the unit open disc, respectively. In particular,
$\Gamma_{y_r}^{-1}(1+z)$ admits a countable set of isolated zeroes on the complex plane, \textit{i.e.}
$\biguplus_{\chi\in G_r}\chi\Z_{\le-1}$, where $G_r$ is the set of solutions of the equation $z^r=(-1)^{r-1}$.

These functions allow to obtain identities, at arbitrary weight, among polyzetas and an analogue situation, as
the ratios $\zeta(2k)/\pi^{2k}$, drawing out consequences about a structure of polyzetas.
This work will be completed, in the forth comming works, by a study a family of functions
obtained as image of rational series for which their linear representation $(\nu,\mu,\eta)$
are such that the Lie algebra generated by the matrices $\{\mu(y)\}_{y\in Y}$ is \textit{solvable}.

\bibliographystyle{amsplain}

%
%
%
%

\end{document}